\documentclass[12pt]{amsart}
\usepackage{entete}

\begin{document}

\begin{abstract}
We link smooth Artin motives to étale local systems and Artin representations. We then construct the ordinary motivic t-structure on Artin motives with integral coefficients and show that the $\ell$-adic realization functor is t-exact.
\end{abstract}
\maketitle
\tableofcontents

\section*{Introduction}
\changelocaltocdepth{0}
Thanks to \cite{thesevoe,ayo07,tcmm,ayo14,em,rob,khan,anabelian}, we now have at our disposal stable $\infty$-categories of mixed étale motives $\mc{DM}_{\et}$ endowed with Grothendieck's six functors formalism and with $\ell$-adic realization functors. 
However, to connect those categories to Grothendieck's conjectural theory of motives, the motivic t-structure is still missing. Its heart would be the abelian category of mixed motivic sheaves which would satisfy Beilinson's conjectures \cite{Jannsen}. Working by analogy with the derived category of $\ell$-adic complexes, there are two possible versions of the motivic t-structure: the \emph{perverse motivic t-structure} and the \emph{ordinary motivic t-structure}. 

This problem seems completely out of reach at the moment. However, when $k$ is a field, the dimensional filtration on the stable $\infty$-category $\mc{DM}_{\et,c}(k,\Q)$ of constructible étale motives over $k$ with rational coefficients was introduced by Beilinson in \cite{beilinson-n-mot}; he showed that the existence of the motivic t-structure implies the existence of t-structures on the subcategory $\mc{DM}^n_{\et,c}(k,\Q)$ of \emph{$n$-motives} which is the subcategory generated by the cohomological motives of proper $k$-schemes of dimension less than $n$. Defining the motivic t-structure on those dimensional subcategories has proved to be easier in some cases.

This question also generalizes to an arbitrary base scheme if we define the subcategory $\mc{DM}^n_{\et,c}(S,\Q)$ of \emph{$n$-motives} in the same fashion. 

The strongest result on this problem concerns the ordinary motivic t-structure and stems from the work of Voevodsky, Orgogozo, Ayoub, Barbieri-Viale, Kahn, Pépin Lehalleur and Vaish.

\begin{theorem*}(\cite{orgo,ayo11,abv,bvk,plh,plh2,vaish}) Let $S$ be a noetherian excellent finite dimensional scheme allowing resolution of singularities by alterations, let $\ell$ be a prime number invertible on $S$ and let $n=0,1$.

Then, there is a t-structure on the stable $\infty$-category $\mc{DM}^n_{\et}(S,\Q)$ which induces a non-degenerate t-structure on the subcategory $\mc{DM}^n_{\et,c}(S,\Q)$ of constructible $n$-motives and such that the $\ell$-adic realization functor
$$\rho_\ell\colon\mc{DM}^n_c(S,\Q)\rar \mc{D}^b_c(S,\Q_\ell)$$
is t-exact when the derived category $\mc{D}^b_c(S,\Q_\ell)$ of constructible $\ell$-adic complexes is endowed with its ordinary t-structure.
\end{theorem*}

The goal of this paper and of its sequel is to generalize this result to integral coefficients in the case of $0$-motives (also known as Artin motives). Recall that Voevodsky showed in \cite[4.3.8]{orange} that, it is not possible to construct a reasonable motivic t-structure on the stable $\infty$-category $\mc{DM}_\mathrm{Nis}(k,\Z)$ of Nisnevich motives for a general base field $k$. However, the motivic t-structure is expected to exist on the stable $\infty$-category $\mc{DM}_{\et}(S,\Z)$ of étale motives.



Let $R$ be a ring of coefficients. The main characters of our text will be the following categories of Artin motives (also known as $0$-motives).
\begin{itemize} \item The category $\DM^A_{\et,c}(S,R)$ of constructible Artin étale motives (or $0$-motives with the above terminology) defined as the thick subcategory (in the sense of \Cref{thickloc}) of the stable $\infty$-category $\DM_{\et}(S,R)$ of étale motives of \cite{ayo14,em} generated by the cohomological motives of finite $S$-schemes.
\item The category $\DM^A_{\et}(S,R)$ of Artin étale motives defined as the localizing subcategory (in the sense of \Cref{thickloc}) of $\DM_{\et}(S,R)$ with the same generators.
\end{itemize}

\subsection*{Rigidity for torsion motives}
One of the main contrasts between Artin étale motives with rational coefficients and Artin étale motives with integral coefficients is that the category of integral motives will contain many motives which should come from higher dimensional motives with torsion coefficients. This is in fact an instance of the rigidity theorem. Let $\Shv(S_{\et},R)$ be the stable $\infty$-category of hypersheaves on the small étale site of $S$ with value in the derived category of $R$-modules. We have by \Cref{small et site 2} a functor
$$\rho_!\colon \Shv(S_{\et},R)\rar \DM^A_{\et}(S,R)$$ which sends a sheaf to the infinite Tate suspension of the $\mb{A}^1$-localization of its extension to the big étale site. This functor is in fact very close to being an equivalence when restricted to torsion objects. More precisely, we define  
\begin{itemize}
    \item The stable category of \emph{torsion étale motives} (resp. \emph{$p^\infty$-torsion étale motives}) denoted by $\mc{DM}_{\et,\mathrm{tors}}(S,R)$ (resp. $\mc{DM}_{\et,p^\infty-\mathrm{tors}}(S,R)$) as the subcategory of the stable category $\mc{DM}_{\et}(S,R)$ made of those objects $M$ such that the motive $M\otimes_\Z \Q$ (resp. $M\otimes_\Z \Z[1/p]$) vanishes.
    \item The stable category $\Shv_{\mathrm{tors}}(S_{\et},R)$ (resp. $\Shv_{p^\infty-\mathrm{tors}}(S_{\et},R)$ ) of \emph{torsion étale sheaves} (resp. \emph{$p^\infty$-torsion étale sheaves}) as the subcategory of the stable category $\Shv(S_{\et},R)$ made of those sheaves $M$ such that the sheaf $M\otimes_\Z \Q$ (resp. $M\otimes_\Z \Z[1/p]$) vanishes.
\end{itemize}
The following result is then a consequence of the Rigidity Theorem \cite[3.2]{bachmanrigidity}.
\begin{proposition*}(\Cref{rigidity V2} and \Cref{torsion motives are Artin})
Let $S$ be a scheme and let $R$ be a good ring (see Notations and Conventions below). Then, 
\begin{enumerate}
    \item Torsion motives are Artin motives.
    \item If $p$ is a prime number, let $j_p\colon S[1/p]\rar S$ be the open immersion. Then, the functor $$\prod_{p \text{ prime}} \Shv_{p^\infty-\mathrm{tors}}(S[1/p]_{\et},R) \rar \mc{DM}_{\et,\mathrm{tors}}(S,R)$$
    which sends an object $(M_p)_{p \text{ prime}}$ to the object $\bigoplus\limits_{p \text{ prime}} j_p^*\rho_! M_p$ is an equivalence.
\end{enumerate}
\end{proposition*}


This result will allow us to construct the ordinary t-structure on torsion motives. One of the main ideas of this text is to deduce information on the t-structures on Artin motives with integral coefficients from the t-structures on Artin motives with rational coefficients and the t-structures on torsion motives.

\subsection*{Smooth Artin motives}
The subcategory of dualizable (or lisse) objects of the derived category of $\ell$-adic sheaves plays a key role in the study of $\ell$-adic sheaves and perverse sheaves. They can be described in terms of local systems and are therefore much more tractable than general perverse sheaves. In this text, the subcategory that will play the role of lisse $\ell$-adic sheaves is the category of smooth Artin motives. 
\begin{itemize}
\item The category $\DM^{smA}_{\et,c}(S,R)$ of constructible smooth Artin étale motives defined as the thick subcategory of $\DM_{\et}(S,R)$ generated by the cohomological motives of finite étale $S$-schemes.
\item The category $\DM^{smA}_{\et}(S,R)$ of smooth Artin étale motives defined as the localizing subcategory of $\DM_{\et}(S,R)$ with the same generators.
\end{itemize}
An important step towards the answer of the questions above is to describe the categories $\DM^{smA}_{\et,(c)}(S,R)$ more precisely. Extending a result of \cite{orgo}, we link them together the more concrete and tractable category of Artin representations of the étale fundamental group of $S$.

The categories of smooth Artin motives can be identified to the following categories.
\begin{itemize}
	\item The \emph{category of lisse étale sheaves} $\Shv_{\mathrm{lisse}}(S,R)$ as the subcategory of dualizable objects of the stable $\infty$-category $\Shv(S_{\et},R)$.
    \item The \emph{category of Ind-lisse étale sheaves} $\Shv_{\In\mathrm{lisse}}(S,R)$ as the localizing subcategory of the stable $\infty$-category $\Shv(S_{\et},R)$ generated by $\Shv_{\mathrm{lisse}}(S,R)$.
\end{itemize}

\begin{theorem*}(\Cref{smooth Artin h-motives}, \Cref{loc cst,description des faisceaux lisses} and \Cref{Artin etale motives over a field}) Let $R$ be a regular good ring and let $S$ be a regular scheme.
\begin{enumerate}
    \item Assume that the residue characteristic exponents of $S$ are invertible in $R$. Then, the functor $\rho_!$ of \Cref{small et site 2} induces monoidal equivalences
$$\Shv_{\In\lis}(S,R)\longrightarrow \DM^{smA}_{\et}(S,R)$$
$$\Shv_{\lis}(S,R)\longrightarrow \DM^{smA}_{\et,c}(S,R).$$
    \item The ordinary t-structure on the stable $\infty$-category $\Shv(S_{\et},R)$ induces t-structures on the subcategories $\Shv_{\lis}(S,R)$ and $\Shv_{\In\lis}(S,R)$. 
    \begin{enumerate}
        \item The heart of the induced t-structure on  $\Shv_{\lis}(S,R)$ is the category $\Loc_S(R)$ of locally constant sheaves of $R$-modules with finitely presented fibers.
        \item The heart of the induced t-structure on  $\Shv_{\In\lis}(S,R)$ is the category $\In\Loc_S(R)$ of filtered colimits of locally constant sheaves of $R$-modules with finitely presented fibers.
    \end{enumerate}
    \item If $S$ is connected, letting $\xi$ be a geometric point of $S$, the fiber functor associated to $\xi$ induces an equivalence of abelian monoidal categories $$\xi^*\colon\mathrm{Loc}_S(R)\rar \Rep^A(\pi_1^{\et}(S,\xi),R)$$
    where $\Rep^A(\pi_1^{\et}(S,\xi),R)$ is the category of Artin representations of the étale fundamental group of $S$ with base point $\xi$.
    \item If $S$ is the spectrum of a field $k$ of characteristic exponent $p$. Then, the functor $\rho_!$ induces monoidal equivalences 
$$\mc{D}(\mathrm{Mod}(G_k,R[1/p]))\simeq \Shv(k_{\et},R[1/p])\longrightarrow \DM_{\et}^{A}(k,R),$$
$$\mc{D}^b(\Rep^A(G_k,R[1/p]))\simeq \Shv_{\lis}(k,R[1/p]) \longrightarrow \DM_{\et,c}^{A}(k,R),$$ where $\Rep^A(G_k,R[1/p])$ 
(resp. $\mathrm{Mod}(G_k,R[1/p])$) is the category of Artin representations (resp. discrete representation) of the absolute Galois group $G_k$ of $k$ with coefficients in $R[1/p]$.
\end{enumerate}

\end{theorem*}

The above theorem shows in particular that the stable $\infty$-category $\DM^{smA}_{\et}(S,R)$ (resp. $\DM^{smA}_{\et,c}(S,R)$) is endowed with a t-structure and identifies the heart with an abelian category of étale sheaves. 
Furthermore, this theorem shows that we can describe the objects of $\DM^A_{\et,c}(S,R)$ as a gluing in the sense of \cite{bbd} of Artin representations. 

The last consequence of the above theorem is a version with integral coefficients of the conservativity conjecture for Artin étale motives (see \Cref{conservativityconj}). It is slightly different from the version with rational coefficients proved in \cite{plh}. Indeed, we need to include all the $v$-adic realization where $v$ runs through the set of valuations on the fraction field of $R$ which are non-negative on $R$ (\textit{i.e.} such that $\{x \in R\mid v(x)>0\}$ defines a prime ideal of $R$). Otherwise any motive of $\ell$-torsion would provide a counterexample to the result.

\subsection*{The ordinary motivic t-structure}
In this paper and its sequel, we will introduce two \emph{homotopy t-structures} on the category $\DM^A_{\et}(S,R)$ whose definitions are inspired by the approach of \cite{plh,vaish,ayo07,bondarko-deglise}. This paper will focus on the \emph{ordinary homotopy t-structure} while the sequel will focus on the \emph{perverse homotopy t-structure}. The t-structures are defined by giving a family of t-non-positive generators. Our goal will be to determine when those t-structures induce t-structures on the category of constructible objects and when the realization functor is t-exact.

We extend the main results of \cite{plh2,vaish} about the ordinary motivic t-structure to the case of integral coefficients.

\begin{theorem*}(\Cref{ordinary is trop bien}) Let $S$ be a scheme allowing resolution of singularities by alterations and let $R$ be a regular good ring. 
\begin{enumerate}\item The ordinary homotopy t-structure induces a non-degenerate t-structure on the subcategory $\DM^A_{\et,c}(S,R)$. We still call this t-structure the \emph{ordinary homotopy} t-structure.
\item Assume that the ring $R$ is a localization of the ring of integers of a number field $K$. Let $v$ be a non-archimedian valuation on $K$ and let $\ell$ be the prime number such that $v$ extends the $\ell$-adic valuation. Then, the reduced $v$-adic realization functor 
$$\overline{\rho}_v\colon \DM_{\et,c}^A(S,R)\rar \mc{D}^b_c(S[1/\ell],R_v)$$ of \Cref{reduced l-adic real} is t-exact when the left hand side is endowed with the ordinary homotopy t-structure and the right hand side is endowed with the ordinary t-structure.
\end{enumerate}
\end{theorem*}

To prove this result, we first study the case of smooth Artin motives and of torsion motives. We show that the inclusion functors $$\DM^{smA}_{\et,c}(S,R)\rar \DM^A_{\et}(S,R)$$
and 
$$\DM_{\mathrm{tors}}(S,R)\rar \DM^A_{\et}(S,R)$$
are t-exact when the left hand sides are endowed with their ordinary t-structures and the right hand side is endowed with the ordinary homotopy t-structure (see \Cref{smooth Artin ordinary V1,ordinary homotopy torsion}).
To finish the proof, we show that the pullback functors are t-exact (see \Cref{f^* ordinary}) and conclude by stratifying any constructible motive by smooth Artin motives.

\subsection*{Acknowledgements}

This paper is part of my PhD thesis, done under the supervision of Frédéric Déglise. I would like to express my deepest gratitude to him for his constant support, for his patience and for his numerous suggestions and improvements. 

My sincere thanks also go to Joseph Ayoub and Jakob Scholbach for their very detailed reports that allowed me to greatly improve the quality of this paper.
I would also like to thank warmly Adrien Dubouloz, Sophie Morel, Riccardo Pengo, Simon Pepin Lehalleur and Jörg Wildeshaus for their help, their kind remarks and their questions which have also enabled me to improve this work.

Finally, I would like to thank Olivier Benoist, Robin Carlier, Mattia Cavicchi, Wiesława Nizioł, Fabrice Orgogozo, Timo Richarz, Wolfgang Soergel, Markus Spitzweck, Olivier Taïbi, Swann Tubach and Olivier Wittenberg for their interest in my work and for some helpful questions and advices. 

\changelocaltocdepth{2}
\section*{Notations and conventions}

In this text we will freely use the language of $\infty$-categories of \cite{htt,ha}.
The term category will by default mean $\infty$-category. When we refer to derived categories, we always refer to the
$\infty$-categorical version.
We will also freely use the language of premotivic and motivic categories of \cite{tcmm}\footnote{A quick introduction (in French) can be found in the first part of my Ph. D. Thesis, available at \url{https://sites.google.com/view/raphael-ruimy/research}.}

All schemes are assumed to be \textbf{noetherian} and of \textbf{finite dimension}; furthermore all smooth (and étale) morphisms and all quasi-finite morphisms are implicitly assumed to be separated and of finite type. 

We let $\Sm$ be the class of smooth morphisms of schemes. For a scheme $S$, we let $S_{\et}$ (resp. $\Sm_S$) be the category of étale (resp. smooth) $S$-schemes.

We adopt the cohomological convention for t-structures (\textit{i.e} the convention of \cite[1.3.1]{bbd} which is the opposite of the convention of \cite[1.2.1.1]{ha}): a t-structure on a stable category $\mc{D}$ is a pair $(\mc{D}^{\leqslant 0},\mc{D}^{\geqslant 0})$ of strictly full\footnote{Recall that a strictly full subcategory is a full subcategory whose set of objects is closed under isomorphisms.} subcategories of $\mc{D}$ having the following properties: 
\begin{itemize}
	\item For any object $M$ of $\mc{D}^{\leqslant 0}$ and any object $N$ of $\mc{D}^{\geqslant 0}$, the abelian group $\pi_0 \Map(M,N[-1])$ vanishes.
	\item We have inclusions $\mc{D}^{\leqslant 0} \subseteq \mc{D}^{\leqslant 0}[-1]$ and $\mc{D}^{\geqslant 0}[-1] \subseteq \mc{D}^{\geqslant 0}$.
	\item For any object $M$ of $\mc{D}$, there exists an exact triangle $M'\rar M \rar M''$ where $M'$ is an object of $\mc{D}^{\leqslant 0}$ and $M''$ is an object of $\mc{D}^{\geqslant 0}[-1]$.
\end{itemize}

A stable category endowed with a t-structure is called a t-category. If $\mc{D}$ is a t-category, we denote by $\mc{D}^\heart=\mc{D}^{\geqslant 0} \cap \mc{D}^{\leqslant 0}$ the heart of the t-structure which is an abelian category and by $\mc{D}^b$ the full subcategory of bounded objects which is a stable t-category.

We will extensively use the results and notations of \cite{em}. We also freely use the language of \cite{tcmm}. Their definitions and results are formulated in the setup of triangulated categories but can readily be adapted to our framework following the ideas of \cite{anabelian,khan,rob,agv}.

If $\mr{D}$ is a premotivic category, for any scheme $S$, we denote by $\un_S$ the unit object of the monoidal category $\mr{D}(S)$.

If $S$ is a scheme and $R$ is a ring, we denote by $\Shv(S_{\et},R)$ the category of étale hypersheaves on the small étale site $S_{\et}$ with value in the derived category of $R$-modules and by $\Sh(S,R)$ its heart which is the category of sheaves of $R$-modules over $S_{\et}$. Furthermore, we denote by $\Shv_{\et}(S,R)$ the category of étale hypersheaves over $\Sm_S$ with value in the derived category of $R$-modules and by $\Sh_{\et}(S,R)$ its heart which is the category of étale sheaves of $R$-modules over $\Sm_S$

There are several models for the category of étale motives. We will use Ayoub's model developped in \cite{ayo07}. If $R$ is a ring, we denote by $\DA_{\et}^{\eff}(S,R)$ the stable subcategory of $\Shv_{\et}(S,R)$ made of its $\AAA$-local objects. Then, the stable category $\DM_{\et}(S,R)$ of \emph{étale motives} over $S$ with coefficients in $R$ is the $\mb{P}^1$-stabilization of the stable category $\DA_{\et}^{\eff}(S,R)$.

Let $S$ be a scheme. Recall that the h-topology is the topology on the category of schemes of finite type over $S$ whose covers are universal topological epimorphisms. The premotivic category of h-motives defined in \cite[5.1.1]{em} is equivalent over noetherian schemes of finite dimension to the premotivic catgeory of étale motives (\cite[5.5.7]{em} applies in this generality, replacing \cite[4.1]{ayo14} by \cite[3.2]{bachmanrigidity} in the proof). Therefore, étale motives satisfy h-descent and we can use the results proved in \cite{em} in the setting of h-motives to prove statements about étale motives.
Let $i\colon Z\rar X$ be a closed immersion and $j\colon U\rar X$ be the complementary open immersion. We call localization triangles the exact triangles of functors: 
\begin{equation}\label{AM.localization}j_!j^*\rar Id \rar i_*i^*.
\end{equation}
\begin{equation}\label{AM.colocalization}i_!i^!\rar Id \rar j_*j^*.
\end{equation}

Let $S$ be a scheme. If $f\colon X\rar S$ is a morphism of finite type, we let $M_S(X)=f_! f^! \mathbbm{1}_S,$ $M_S^{\BM}(X,R)=f_!\un_X$ and $h_S(X,R)=f_*\un_X.$
Again, when the ring $R$ is non-ambiguous, those functors will simply be denoted by $M_S^{\BM}$ and $h_S$.

\textbf{All rings are assumed to be commutative and noetherian}. A ring is said to be \emph{good} if it is either a localization of the ring of integers of a number field or noetherian and of positive characteristic.

If $S$ is a scheme. A \emph{stratification} of $S$ is a partition of $S$ into non-empty equidimensional locally closed subschemes called \emph{strata} such that the topological closure of any stratum is a union strata.

If $S$ is a scheme and $\xi$ is a geometric point of $S$, we denote by $\pi_1^{\et}(S,\xi)$ the étale fundamental group of $S$ with base point $\xi$ defined in \cite{sga1}.

If $x$ is a point of a scheme $S$, we denote by $k(x)$ the residue field of $S$ at the point $x$.

If $k$ is a field, we denote by $G_k$ its absolute Galois group.

We denote by $\Z$ the ring of integers and by $\Q$ the field of rational numbers. If $p$ is a prime number, we denote by $\mb{F}_p$ the field with $p$ elements, by $\mb{Z}_p$ the ring of $p$-adic integers, by $\Q_p$ the field of $p$-adic numbers and by $\overline{\Q}_p$ the latter's algebraic closure. Finally, if $\pp$ is a prime ideal of the ring $\Z$ we denote by $\Z_\pp$ the localization of the ring $\Z$ with respect to $\pp$.

If $R$ is a localization of the ring of integers of a number field $K$ and if $v$ is a valuation on $K$, we denote by $R_v$ the completion of $R$ with respect to $v$. If the valuation $v$ is non-negative on $R$, we denote by $R_{(v)}$ the localization of $R$ with respect to the prime ideal $\{ x\in R\mid v(x)>0\}$.

\begin{remark*} The assumption that all schemes are noetherian and of finite dimension is absolutely essential in the proof of the continuity property and of the absolute purity property (among others) for étale motives. In the case of Nisnevich motives, one can have the continuity over quasi-compact quasi-separated schemes (see \cite[C]{hoy}). In fact, as over noetherian schemes of finite dimension, étale and Nisnevich motives with rational coefficients coincide, some of our results can be extended to this setting.
\end{remark*}

\section{Preliminaries}

\subsection{Categorical Preliminaries}
\begin{definition}\label{thickloc} Let $\mc{C}$ be a stable category 
\begin{enumerate}\item A \emph{thick} subcategory of $\mc{C}$ is a full subcategory $\mc{D}$ of $\mc{C}$ which is closed under finite limits, finite colimits and retracts.
\item A \emph{localizing} subcategory of $\mc{C}$ is a full subcategory $\mc{D}$ of $\mc{C}$ which is closed under finite limits and arbitrary colimits.
\item Let $\mc{E}$ be a set of objects. We call \emph{thick} (resp. \emph{localizing}) \emph{subcategory generated by} $\mc{E}$ the smallest thick (resp. localizing) subcategory of $\mc{C}$ whose set of objects contains $\mc{E}$.
	\end{enumerate}
\end{definition}

\subsubsection{Complements on t-categories}
We first recall the following definition from \cite[1.3.19]{bbd} that we will use constantly.
\begin{definition}
Let $\mc{D}$ be a t-category and $\mc{D}'$ be a full stable subcategory of $\mc{D}$. If  $(\mc{D}^{\leqslant 0} \cap \mc{D}', \mc{D}^{\geqslant 0} \cap \mc{D}')$ defines a t-structure on $\mc{D}'$, we say that $\mc{D'}$ is a \emph{sub-t-category} of $\mc{D}$, that the t-structure of $\mc{D}$ \emph{induces a t-structure} on $\mc{D}'$ and call the latter the \emph{induced t-structure}.
\end{definition}
The following lemma will allow us to define the ordinary t-structure.
\begin{proposition}(\cite[1.4.4.11]{ha}) \label{AM.t-structure generated} Let $\mc{C}$ be a presentable stable category. Given a small family $\mc{E}$ of objects, the smallest subcategory $\mc{E}_-$ closed under small colimits and extensions is the set of non-positive objects of a t-structure.
\end{proposition}

In this setting, we will call this t-structure the \emph{t-structure generated by} $\mc{E}$. When a set $\mc{E}$ of objects has a small set of isomorphism classes, we still call \emph{t-structure generated by} $\mc{E}$ the t-structure generated by a small family of representatives of $\mc{E}$.

\begin{proposition}\label{generatorze} Let $\mc{C}$ be a category endowed with a Grothendieck topology $\tau$ and let $\Lambda$ be a ring. Denote by $\Shv_\tau(\mc{C},\Lambda)$ the t-category of $\Lambda$-modules in the category of $\tau$-hypersheaves of spectra and denote by $$\Lambda(-)\colon \mc{C} \rar \Shv_\tau(\mc{C},\Lambda)$$ the $\tau$-hypersheafification composed with the Yoneda functor. Then, 
\begin{enumerate}
    \item The category $\Shv_\tau(\mc{C},\Lambda)$ is generated as a localizing subcategory of itself by the sheaves of the form $\Lambda(X)$ where $X$ runs through the set of objects of $\mc{C}$.
    \item Assume that $\mc{C}$ has a small set of isomorphism classes. Then, the t-structure on $\Shv_\tau(\mc{C},\Lambda)$ is generated by the sheaves of the form $\Lambda(X)$ where $X$ runs through the set of objects of $\mc{C}$.
\end{enumerate}
\end{proposition}
\begin{proof} Since $\Shv_\tau(\mc{C},\Lambda)$ is a localization of the category $\PShv(\mc{C},\Lambda)$ of presheaves of $\Lambda$-modules and since every t-non-positive sheaf is the sheafification of a t-non-positive presheaf it suffices to prove the analogous result for the category of presheaves. 

The category $\PShv(\mc{C},\Lambda)$ is compactly generated by the sheaves of the form $\Lambda(X)$ where $X$ runs through the set of objects of $\mc{C}$. The first assertion then follows from \cite[8.4.1]{neeman}.

We now prove the second assertion. A presheaf $M$ is t-positive with respect to the ordinary t-structure on $\PShv(\mc{C},\Lambda)$ if and only if for object $X$ of $\mc{C}$, the complex $M(X)$ of $\Lambda$-modules is $0$-connected. But since for any object $X$ of $\mc{C}$, we have $$M(X)=\Map_{\PShv(\mc{C},\Lambda)}(\Lambda(X),M),$$
this is equivalent to the condition that $M$ is t-positive with respect to the t-structure generated by the sheaves of the form $\Lambda(X)$ with $X$ an object of $\mc{C}$. Hence, both t-structures coincide.  
\end{proof}

The following lemma shows how t-structures generated by a set of objects behave with respect to change of coefficients.

\begin{lemma}\label{coeff change} Let $\mc{C}$ be $\Z$-linear presentable stable category, let $\mc{E}$ be a small family of objects of $\mc{C}$, let $n$ be an integer and let $A$ be a localization of $\Z$. Denote by $\sigma_n\colon \Z\rar \Z/n\Z$ and by $\sigma_A\colon \Z\rar A$ the natural ring morphisms.

If $R=\Z/n\Z$ or $R=A$ and if $\sigma\colon \Z\rar R$ is the natural ring morphism, we denote by $$\sigma^*\colon \mc{C} \leftrightarrows \mc{C}\otimes_{\mathrm{Mod}_\Z} \mathrm{Mod}_{R}\colon \sigma_*$$ the adjunction given by \cite[4.6.2.17]{ha}. 
Denote by $\mc{E}_R$ the family of objects of $\mc{C}\otimes_{\mathrm{Mod}_\Z} \mathrm{Mod}_{R}$ made of the $\sigma^*M$.

Endow $\mc{C}$ (resp. $\mc{C}\otimes_{\mathrm{Mod}_\Z} \mathrm{Mod}_{\Z/n\Z}$, resp. $\mc{C}\otimes_{\mathrm{Mod}_\Z} \mathrm{Mod}_{A}$) with the t-structure generated by $\mc{E}$ (resp. $\mc{E}_{Z/n\Z}$, resp. $\mc{E}_A$) in the sense of \Cref{AM.t-structure generated}. 

\begin{enumerate}\item The functors of the form $\sigma_*$ are t-exact.
\item Assume furthermore that if $M$ belongs to $\mc{E}$ and $N$ is an object of $\mc{C}$, the map $$\sigma_A^*\Map_{\mc{C}}(M,N)\rar \Map_{\mc{C}}(M,\sigma_A^*N)$$ is an equivalence. Then, the functor $\sigma_A^*$ is t-exact.
\end{enumerate}    
\end{lemma}
\begin{proof}
Let $R=\Z/n\Z$ or $R=A$ and let $\sigma\colon \Z\rar R$ be the natural ring morphism. The colimit preserving exact functor $\sigma^*$ sends the generators of the t-structure to t-non-positive objects and is therefore right t-exact. By adjunction, the functor $\sigma_*$ is therefore left t-exact.

Notice then that if $M$ lies in $\mc{E}$, the object $\sigma_n^*M$ is the cofiber of the map $M\overset{\times n}{\rar} M.$
Therefore the object $\sigma_n^*M$ is a colimit of t-non-positive objects of $\mc{C}$ and is therefore t-non-positive. 

Since the subcategory of t-non-positive objects of the stable category $\mc{C}\otimes_{\mathrm{Mod}_\Z} \mathrm{Mod}_{\Z/n\Z}$ is the smallest subcategory closed under small colimits and extensions containing the $M\otimes_\Z \Z/n\Z$ when $M$ runs through $\mc{E}$, and since the exact functor $(\sigma_n)_*$ preserves small colimits, the functor $(\sigma_n)_*$ is therefore right t-exact. 

On the other hand, write $A=\Sigma^{-1}\Z$ with $\Sigma$ a multiplicative system. Then, the object $\sigma_A^*M$ is the colimit of the diagram of value $M$ with transition maps the multiplications by $s$ for $s \in \Sigma$. Hence, it is t-non-positive. Since the functor $(\sigma_A)_*$ is also colimit preserving, this yields its t-exactness.

To prove the second assertion, it suffices to show that the exact functor $\sigma_A^*$ is left-t-exact under our assumption. Take $N$ a t-non-negative object of $\mc{C}$. Let $M$ be an object of $\mc{E}$. Then, we have $$\Map_{\mc{C}\otimes_{\mathrm{Mod}_\Z} \mathrm{Mod}_A}(\sigma_A^*M,\sigma_A^*N)=\Map(M,\sigma_A^*N)$$

But, by assumption, we have $$\Map(M,\sigma_A^*N)=\sigma_A^*\Map(M,N).$$

Thus, the chain complex $\Map(\sigma_A^*M,\sigma_A^*N)$ is $(-1)$-connected and the object $\sigma_A^*(N)$ is t-non-negative. Therefore, the functor $\sigma_A^*$ is t-exact.
\end{proof}

In our study, we will need the following:
\begin{lemma}\label{cons family of t-ex functors} Let $I$ be a small set, let $\mc{D}$ be a t-category. For any index $i$ in $I$, let $\mc{D}_i$ be a t-category and let $T_i\colon \mc{D}\rar \mc{D}_i$ be a t-exact functor. Assume that the family $(T_i)_{i \in I}$ is conservative. 

Then, an object $M$ of $\mc{D}$ is t-positive (resp. t-negative) if and only if for any index $i$ in $I$, the object $T_i(M)$ is t-positive (resp. t-negative) in $\mc{D}_i$.
\end{lemma}
\begin{proof} The t-exactness of the $T_i$ yields the "only if" part of the lemma. 

We now prove the "if" part. Let $M$ be an object of $\mc{D}$ such that for any $i\in I$, the object $T_i(M)$ of $\mc{D}_i$ is t-positive. Then, for any $i\in I$, the functor $T_i$ sends the map  $M \rar \tau_{>0}M$ to an equivalence since t-exact functors commute with the positive truncation. Therefore, the map $M \rar \tau_{>0}M$ is an equivalence and $M$ is t-positive. The t-negative case is similar.
\end{proof}

\subsubsection{Idempotent Complete $\mathfrak{p}$-localization}\label{p-loc}
We recall some notions from \cite[B]{em}.

\begin{definition} Let $\mc{C}$ be a $\Z$-linear stable category and let $\mathfrak{p}$ be a prime ideal of $\Z$. 
\begin{enumerate} \item The \emph{naive $\mathfrak{p}$-localization} $\mc{C}_\mathfrak{p}^\mathrm{na}$ of $\mc{C}$ is the category with the same objects as $\mc{C}$ and such that if $M$ and $N$ are objects of $\mc{C}$, we have $$\Map_{\mc{C}_\mathfrak{p}^\mathrm{na}}(M,N)=\Map_\mc{C}(M,N)\otimes_\Z \Z_{\mathfrak{p}}.$$
\item The $\mathfrak{p}$-localization $\mc{C}_{\mathfrak{p}}$ of $\mc{C}$ is the idempotent completion of the naive $\mathfrak{p}$-localization $\mc{C}_\pp^{\mathrm{na}}$ of $\mc{C}$.
\end{enumerate} 
\end{definition}
If $\mc{C}$ is a $\Z$-linear stable category and if $\mathfrak{p}$ is a prime ideal of $\Z$, we have a canonical exact functor $$\mc{C}\rar \mc{C}_{\mathfrak{p}}.$$

\begin{proposition}\label{p-loc 1} Let $\mc{C}$ be a $\Z$-linear stable category and let $\mathfrak{p}$ be a prime ideal of $\Z$. 
\begin{enumerate}
    \item The category $\mc{C}_\mathfrak{p}$ is universal among the $\Z_\mathfrak{p}$-linear idempotent complete stable categories under $\mc{C}$.
    \item Assume that the stable category $\mc{C}$ is endowed with a t-structure $(\mc{C}^{\leqslant 0},\mc{C}^{\geqslant 0})$. 
    \begin{enumerate}
        \item Let $\mc{C}^{\leqslant 0,\na}_\mathfrak{p}$ (resp. $\mc{C}^{\geqslant 0,\na}_\mathfrak{p}$) be the full subcategory of $\mc{C}_\mathfrak{p}^\na$ made of those objects lying in the essential image of the functor $\mc{C}^{\leqslant 0}\rar \mc{C}_{\mathfrak{p}}^\na$ 
        (resp. of the functor $\mc{C}^{\geqslant 0}\rar \mc{C}_{\mathfrak{p}}^\na$). Then, the pair $(\mc{C}^{\leqslant 0,\na}_{\mathfrak{p}},\mc{C}^{\geqslant 0,\na}_{\mathfrak{p}})$ defines a t-structure on the stable category $\mc{C}_{\mathfrak{p}}^\na$ which is the unique t-structure such that the canonical functor $\mc{C}\rar \mc{C}_{\mathfrak{p}}^\na$ is t-exact.
        \item Let $\mc{C}^{\leqslant 0}_\mathfrak{p}$ (resp. $\mc{C}^{\geqslant 0}_\mathfrak{p}$) be the full subcategory of $\mc{C}_\mathfrak{p}$ made of those objects which are retract of objects of  $\mc{C}^{\leqslant 0,\na}$ (resp. $\mc{C}^{\geqslant 0,\na}\rar \mc{C}_{\mathfrak{p}}$). Then, the pair $(\mc{C}^{\leqslant 0}_{\mathfrak{p}},\mc{C}^{\geqslant 0}_{\mathfrak{p}})$ defines a t-structure on the stable category $\mc{C}_{\mathfrak{p}}$ which is the unique t-structure such that the canonical functor $\mc{C}\rar \mc{C}_{\mathfrak{p}}$ is t-exact.
    \end{enumerate} 
\end{enumerate} 
\end{proposition}
\begin{proof} The first assertion follows from the proof of \cite[B.1.6]{em}. The second assertion follows from the proofs of \cite[B.2.1]{em} and \cite[B.2.2]{em}.
\end{proof}

If $\mc{C}$ and $\mc{D}$ are $\Z$-linear stable categories, if $F\colon \mc{C}\rar \mc{D}$ is an exact functor and if $\mathfrak{p}$ is a prime ideal of $\Z$, we have canonical exact functors 
$$F_\mathfrak{p}^\na\colon \mc{C}_\mathfrak{p}^\na\rar \mc{D}_\mathfrak{p}^\na.$$
$$F_\mathfrak{p}\colon \mc{C}_\mathfrak{p}\rar \mc{D}_\mathfrak{p}.$$

\begin{proposition}\label{p-loc 2} Let $\mc{C}$ and $\mc{D}$ be $\Z$-linear stable categories, let $F\colon\mc{C}\rar \mc{D}$ be an exact functor and let $\mathfrak{p}$ be a prime ideal of $\Z$. 
\begin{enumerate}
    \item If the functor $F$ is fully faithful, so are the functors $F_\pp^\na$ and $F_{\mathfrak{p}}$.
    \item If the stable categories $\mc{C}$ and $\mc{D}$ are endowed with t-structures and if the functor $F$ is right (or left) t-exact, so are the functors $F_\pp^\na$ and $F_{\mathfrak{p}}$ when the stable categories $\mc{C}_{\mathfrak{p}}^\na$, $\mc{D}_{\mathfrak{p}}^\na$, $\mc{C}_{\mathfrak{p}}$ and $\mc{D}_{\mathfrak{p}}$ are endowed with the t-structures defined in \Cref{p-loc 1}.
\end{enumerate}
\end{proposition}

\begin{proposition}\label{p-locality}(\cite[B.1.7]{em}) Let $\mc{C}$ be a $\Z$-linear stable category and let $\mc{D}$ be a thick stable subcategory of $\mc{C}$. An object $M$ of $\mc{C}$ belongs to $\mc{D}$ if and only if the image of $M$ in $\mc{C}_\mathfrak{p}$
belongs to $\mc{D}_\pp$ for any maximal ideal $\mathfrak{p}$ of $\Z$.
\end{proposition}
\subsection{Stable Categories of Étale motives}

\subsubsection{Small \'Etale Site}\label{small et site} We have a chain of premotivic adjunctions
$$\begin{tikzcd}
\Shv_{\et}(-,R)\ar[r,"L_\AAA",shift left=1]& \DA_{\et}^{\eff}(-,R)\ar[r,"\Sigma^\infty",shift left=1]\ar[l,"\iota",shift left=1]& \DM_{\et}(-,R)\ar[l,"\Omega^\infty",shift left=1]
\end{tikzcd}.$$

 We can add a pair of adjoint functors to left of this diagram, which is not premotivic in general. The inclusion of sites $\rho\colon S_{\text{ét}} \rightarrow \Sm_S$ indeed provides a pair of adjoint functors
$$\rho_\#\colon \Shv(S_{\text{ét}}, R) \rightleftarrows \Shv_{\text{ét}}(\Sm_S, R)\colon \rho^*$$
This extends to a morphism between objects in $\Fun(\mathrm{Sch}^\text{op}, \CAlg(\PrL_\Stb))$.

We still denote by $$R_S\colon S_{\et}\rar \Shv(S_{\et},R)$$ 
the Yoneda embedding. This will be harmless as the sheaf on $S_{\et}$ representing an étale $S$-scheme $X$ is send by the functor $\rho_\#$ to the sheaf on $\Sm_S$ which represents $X$.

If $R$ is an $n$-torsion ring, where the integer $n$ is invertible on $S$, then the functor $L_{\AAA}\rho_\#$ and therefore, the functor $\rho_!=\Sigma^\infty L_{\AAA}\rho_\#$ is a monoidal equivalence (it is the rigidity theorem \cite[3.2]{bachmanrigidity}).

\subsubsection{Change of Coefficients}\label{Change of coefficients} Let $\mr{D}$ be one of the fibered categories $\DM_{\et}$, $\DA^{\eff}_{\et}$, $\Shv_{\et}$ and $\Shv(-_{\et})$.
Let $\sigma\colon R \rar R'$ be a ring homomorphism. 

The adjunction 
$$\sigma^* \colon \mathrm{Mod}_R\leftrightarrows \mathrm{Mod}_{R'} \colon \sigma_*$$
provides a premotivic adjunction 
$$\sigma^* \colon \mr{D}(-,R)\leftrightarrows \mr{D}(-,R') \colon \sigma_*,$$  except for $\mr{D}=\Shv(-_{\et})$ where the adjunction exists but is not premotivic.

If $S$ is a scheme, we can identify $\mr{D}(S,R)$ with the category of $R$-modules in $\mr{D}(S,\Z)$. The functor $\sigma_*$ is then the forgetful functor, and the functor $\sigma^*$ is the functor $-\otimes_R R'$.

Finally, the map $\sigma\mapsto \sigma^*$ comes from the functor 
$$\mathrm{CAlg}(\mbf{Ab})\xrightarrow{H}\mathrm{CAlg}(\mathcalboondox{Sp})\xrightarrow{h(S)\otimes \id}\mathrm{CAlg}(\widehat{\Sm_S}\otimes \mathcalboondox{Sp})\xrightarrow{L_{\AAA}L_{\et}}\mathrm{CAlg}(\mc{SH}(S))\xrightarrow{\mathrm{Mod}_{-}}\CAlg(\PrL_\Stb).$$


\begin{proposition}\label{commutation} Let $\mr{D}$ be one of the fibered categories $\DM_{\et}$, $\DA^{\eff}_{\et}$ and $\Shv_{\et}$. Let $\sigma\colon R\rar R'$ be a ring morphism of the form $\sigma_A\colon R\rar R\otimes_\Z A$ or $\sigma_n\colon R\rar R/nR$, where $A$ is a localization of $\Z$ and $n$ is an integer.
\begin{enumerate}
    \item The functor $\sigma^*$ commutes with the functors $\rho_\#$, $\rho^*$, $L_{\AAA}$, $\iota$, $\Sigma^\infty$ and $\Omega^\infty$ defined in \Cref{small et site}.
    \item Let $n$ be an integer, let $S$ be a scheme and let $M$ and $N$ be objects of $\mr{D}(S,R)$. Then, the canonical maps
$$\sigma_n^*\Hom_{\mr{D}(S,R)}(M,N) \rar \Hom_{\mr{D}(S,R)}(M,\sigma_n^*N)$$
$$\sigma_n^*\Map_{\mr{D}(S,R)}(M,N) \rar \Map_{\mr{D}(S,R)}(M,\sigma_n^*N)$$
are equivalences.
    \item Let $A$ be a localization of $\Z$, let $S$ be a scheme and let $M$ and $N$ be objects of $\mr{D}(S,R)$. Assume that $M$ is constructible in the sense that it is in the thick subcategory generated by representable objects (and Tate twist in the case of non-effective étale motives). Then, the canonical maps
$$\sigma_A^*\Hom_{\mr{D}(S,R)}(M,N) \rar \Hom_{\mr{D}(S,R)}(M,\sigma_A^*N)$$
$$\sigma_A^*\Map_{\mr{D}(S,R)}(M,N) \rar \Map_{\mr{D}(S,R)}(M,\sigma_A^*N)$$
are equivalences.
    \item If $\mr{D}$ is one of the categories $\DM_{\et}$, $\DA^{\eff}_{\et}$ and $\Shv_{\et}$, the functor $\sigma^*$ commute with the six functors formalism.
\end{enumerate}
\end{proposition}
\begin{proof} The functor $\rho^*$ is the restriction to the small étale site and therefore, it commutes with the forgetful functor $\sigma_*$ and with the functor $\sigma^*$. Hence, its left adjoint $\rho_\#$ also commutes with the functor $\sigma^*$.

The functor $\iota$ is the inclusion of $\AAA$-local sheaves into the fibered category $\Shv_{\et}$. Since the $\AAA$-locality property does not depend on the ring of coefficients, the functor $\iota$ commutes with the forgetful functor $\sigma_*$. Therefore the functors $L_{\AAA}$ and $\sigma^*$ commute. Furthermore, \cite[5.4.5, 5.4.11]{em} ensure that if $M$ is $\AAA$-local, then so is $\sigma^*(M)$ which implies that $\iota$ and $\sigma^*$ commute.

Following \cite[5.4]{em} the functor $\sigma^*$ commutes with geometric sections. Since the functors $\sigma^*$ and $\Sigma^\infty$ are colimit-preserving and since by \Cref{generatorze}, the category $\DA^{\eff}_{\et}(S,R)$ is generated as a localizing subcategory of itself by geometric sections, those functors commute. The proof of \cite[5.5.5]{em} shows that the functor $\Omega^\infty$ commutes with the functor $\sigma_A^*$. Finally, we have exact triangle of functors
$$\id\overset{\times n}{\rar} \id \rar \sigma_n^*$$
in $\DM_{\et}$ and $\DA^{\eff}_{\et}$
which induce an exact triangles of functors when applying $\Omega^\infty$ to the left or to the right
$$\Omega^\infty\overset{\times n}{\rar} \Omega^\infty \rar \Omega^\infty\sigma_n^*$$
$$\Omega^\infty\overset{\times n}{\rar} \Omega^\infty \rar \sigma_n^*\Omega^\infty$$

Hence, the functors $\Omega^\infty$ and $\sigma_n^*$ commute, which finishes the proof of the first assertion.

The second assertion follows from the exactness of the triangle $$\id\overset{\times n}{\rar} \id \rar \sigma_n^*$$ in the categories $\mr{D}(S,R)$ and $\mc{D}(R)$. 

In the case where $\mr{D}=\DM_{\et}$, the third assertion follows from \cite[5.4.11]{em}. Tensoring with $A$ respects $\AAA$-local objects by \cite[1.1.11]{em}. Therefore, the third assertion in the case where $\mr{D}=\DA_{\et}^{\eff}$ follows from the same assertion when $\mr{D}=\Shv_{\et}$. The assertion in both of the cases when $\mr{D}=\Shv_{\et}$ and when $\mr{D}=\Shv(-_{\et})$ follows from \cite[1.1.11]{em}.

Using \cite[5.4.5, 5.4.11, A.1.1.16]{em}, we see that the six functors commute with $\sigma^*$ in the case of $\DM_{\et}$. We can mimic the proofs of \cite[5.4.5, 5.4.11]{em} in the other cases to prove the last assertion. 
\end{proof}
\begin{lemma}\label{conservativity}(\cite[5.4.12]{em}) Let $\mr{D}$ be one of the fibered categories $\DM_{\et}$, $\DA^{\eff}_{\et}$, $\Shv_{\et}$ and $\Shv(-_{\et})$. Let $S$ be a scheme and let $R$ be a good ring, then with the above notations, the family $(\sigma_\Q^*,\sigma_p^* \mid p \text{ prime })$ is conservative.
\end{lemma}

Recall the category $\DM_{\et,c}(S,R)$ of constructible motives defined as the thick subcategory generated by the $M_S(X)(n)$ for $X$ smooth over $S$.
\begin{proposition}\label{p-loc cons 1} Let $\pp$ be a prime ideal of $\Z$ and let $S$ be a scheme. Then, with the notations of \Cref{p-loc} the functor $$\DM_{\et,c}(S,R)_\pp\rar \DM_{\et,c}(S,R\otimes_\Z \Z_\pp)$$ is an equivalence.
\end{proposition}
\begin{proof} With the notations of \Cref{p-loc}, the functor $$\DM_{\et,c}(S,R)_\pp^\mathrm{na}\rar \DM_{\et,c}(S,R\otimes_\Z \Z_\pp)$$ is fully faithful by \Cref{commutation}. Since the category $\DM_{\et,c}(S,R\otimes_\Z \Z_\pp)$ is idempotent complete, the functor $$\DM_{\et,c}(S,R)_\pp\rar \DM_{\et,c}(S,R\otimes_\Z \Z_\pp)$$ is therefore fully faithful. Its essential image is therefore a thick subcategory which contains the generators of $\DM_{\et,c}(S,R\otimes_\Z \Z_\pp)$. Therefore, it is an equivalence.
\end{proof}


We finish this section with an analog of \Cref{commutation}(2) for infinite torsion. If $p$ is a prime number denote by $\Z(p^\infty)=\Z[1/p]/\Z$ the Prüfer $p$-group.

\begin{proposition}\label{torsion commutation} Keep the same notations as in \Cref{commutation}. Let $S$ be a scheme and let $M$ and $P$ be objects of $\DM_{\et}(S,R)$. Assume that $M$ is constructible.

\begin{enumerate}
\item The canonical functor $$\Map_{\DM_{\et}(S,R)}(M,P)\otimes_\Z \Q/\Z \rar \Map_{\DM_{\et}(S,R)}(M,P\otimes_\Z \Q/\Z)$$ is an equivalence.
\item Let $p$ be a prime number. The canonical functor $$\Map_{\DM_{\et}(S,R)}(M,P)\otimes_\Z \Z(p^\infty) \rar \Map_{\DM_{\et}(S,R)}(M,P\otimes_\Z \Z(p^\infty))$$ is an equivalence.
\item The canonical functor $$\bigoplus \limits_{p\text{ prime}} \Map_{\DM_{\et}(S,R)}(M,P\otimes_\Z \Z(p^\infty)) \rar \Map_{\DM_{\et}(S,R)}(M,P\otimes_\Z \Q/\Z)$$ is an equivalence.
\item The canonical functor $$\colim_n \Map(M,\sigma_n^*(P)) \rar  \Map(M,\colim_n \sigma_n^*(P))$$ is an equivalence.
\item Let $p$ be a prime number. The canonical functor $$\colim_n \Map(M,\sigma_{p^n}^*(P)) \rar  \Map(M,\colim_n \sigma_{p^n}^*(P))$$ is an equivalence.
\end{enumerate}
\end{proposition}
\begin{proof} We have an exact triangle $$P\rar P \otimes_\Z \Q \rar P \otimes_\Z \Q/\Z.$$

Thus, we get a morphism of exact triangles
$$
\begin{tikzcd}
{\Map(M,P)} \arrow[d] \arrow[r] & {\Map(M,P) \otimes_\Z \Q} \arrow[d] \arrow[r] & {\Map(M,P)\otimes_\Z \Q/\Z} \arrow[d] \\
{\Map(M,P)} \arrow[r]           & {\Map(M,P \otimes_\Z \Q)} \arrow[r]           & {\Map(M,P\otimes_\Z \Q/\Z))}          
\end{tikzcd}.
$$

Since the left and the middle vertical arrows are equivalences by \Cref{commutation}, the right vertical arrow is an equivalence which yields Assertion (1).

Similarly, we have an exact triangle $$P\rar P \otimes_\Z \Z[1/p] \rar P \otimes_\Z \Z(p^\infty).$$

Thus, we get a morphism of exact triangles:
$$
\begin{tikzcd}
{\Map(M,P)} \arrow[d] \arrow[r] & {\Map(M,P) \otimes_\Z \Z[1/p]} \arrow[d] \arrow[r] & {\Map(M,P)\otimes_\Z \Z(p^\infty)} \arrow[d] \\
{\Map(M,P)} \arrow[r]           & {\Map(M,P \otimes_\Z \Z[1/p])} \arrow[r]           & {\Map(M,P\otimes_\Z \Z(p^\infty)))}          
\end{tikzcd}.
$$

Since the left and the middle vertical arrows are equivalences by \Cref{commutation}, the right vertical arrow is an equivalence which yields Assertion (2).

Since $$\Q/\Z=\bigoplus \limits_{p\text{ prime}} \Z(p^\infty),$$ Assertion (3) follows from Assertions (1) and (2).

To prove Assertion (4), observe we have a commutative diagram
$$\begin{tikzcd}
{\colim_n \Map(M,\sigma^*_n(P))} \arrow[d] \arrow[r] & {\Map(M,\colim_n \sigma_n^*(P))} \arrow[dd] \\
{\colim_n \sigma_n^* \Map(M,P)} \arrow[d] &                                             \\
{\Map(M,P)\otimes_R (R\otimes_\Z \Q/\Z)} \arrow[r]   & {\Map(M,P \otimes_R (R \otimes_\Z \Q/\Z))} 
\end{tikzcd}$$
where the right vertical arrow and the bottom left vertical arrow are equivalences because tensor products and colimits commute, the upper left vertical arrow is an equivalence by \cite[5.4.5]{em} and the bottom horizontal arrow is an equivalence by Assertion (1). This yields Assertion (4).

We can derive Assertion (5) from Assertion (2) in a similar fashion.
\end{proof}

\subsection{The \texorpdfstring{$\ell$}{\textell}-adic Realization and the Bhatt-Scholze Formalism}

In this section, we recall the $\ell$-adic formalism.

Let $S$ be a scheme and $\ell$ be a prime number invertible on $S$. Let $\Lambda$ be a finite extension of $\Q_\ell$ or the ring of integers of a finite extension of $\Q_\ell$. There are several definitions of the \emph{derived category of constructible $\ell$-adic sheaves} $\mc{D}^b_c(S,\Lambda)$. The first definition was given in \cite{bbd} for $S$ of finite type over a field. In \cite{torsten}, \cite{ayo14}, \cite[7.2]{em}, \cite{bhatt-scholze} and \cite[XIII.4]{travauxgabber}, this definition was extended to all quasi-compact quasi-separated schemes. In this section, we recall the definitions of \cite[7.2]{em} and \cite{bhatt-scholze} and explain the link between them. We also recall some usual properties of the category $\mc{D}^b_c(S,\Lambda)$.
\subsubsection{The \texorpdfstring{$\lambda$}{λ}-adic realization}
We now fix a discrete valuation ring $R$ with a local parameter $\lambda$ such that the residue field $R/\lambda R$ is finite. We denote by $\ell$ the latter's characteristic.

\begin{definition}
Let $\mc{C}$ be a stable $R$-linear category. An object $M$ in $\mc{C}$ is $\lambda$-complete if the natural map 
$$M \rar \lim M\otimes_R (R/\lambda^n R)$$
is an equivalence. We denote by $\mc{C}_\lambda$ the full subcategory of $\mc{C}$ formed by $\lambda$-complete objects.
\end{definition}

We have a $\lambda$-completion functor $\rho_\lambda^0\colon \mc{C} \rar \mc{C}_\lambda$, which sends an object $M$ to the object $\lim M\otimes_R (R/\lambda^n R)$. 




If $S$ is a scheme, we let $\DM_{\et}(-,R_\lambda):=\DM_{\et}(-,R)_\lambda$.
\begin{proposition}\label{rhoell1}(\cite[7.2.11]{em})
The premotivic category $\DM_{\et}(-,R_\lambda)$ can be endowed with the six functors formalism and satisfies absolute purity. Furthermore, the functor $\rho_\lambda^0$ commutes with the six functors.
\end{proposition}

We have the following variant of the rigidity theorem according to \cite[2.10.4]{agv}.

\begin{proposition} 
When restricted to the category of $\Z[1/\ell]$-schemes, the adjunction 
$$\rho_!\colon \Shv(-_{\et},R)_\lambda\leftrightarrows\DM_{\et}(S,R_\lambda)\colon \rho^!$$
is an equivalence.
\end{proposition}

Finally, we have a constructible version of this variant of the rigidity theorem. We can indeed define the category of constructible $\lambda$-complete objects as follows.

\begin{definition}
Let $S$ be a scheme. A $\lambda$-complete étale sheaf (resp. étale motive) is \emph{constructible} if there exists a stratification of $X$ such that the restriction of the sheaf to each stratum is dualizable. We denote by $\Shv_{c}(S_{\et},R_\lambda)$ (resp. $\DM_{\et,c}(S,R_\lambda)$) the full subcategory spanned by constructible objects.
\end{definition}

We immediatly conclude that over $\Z[1/\ell]$-schemes, the functor $\rho_!$ induces an equivalence of monoidal categories 
$$\Shv_{c}(S_{\et},R_\lambda)\rar\DM_{\et,c}(S,R_\lambda).$$

We now denote by $R_\lambda$ the completion of $R$ with respect to the $\lambda$-adic valuation. We now present the Bhatt-Scholze adic formalism in the version developed in \cite{hrs}.

Bhatt and Scholze defined the pro-étale topology as follows: a morphism of schemes $f\colon X\rar Y$ is \emph{weakly étale} if it is flat and if the diagonal morphism $\Delta_f\colon X\rar X\times_Y X$ is flat. The pro-étale site $X_{\proet}$ of a scheme $X$ is then defined as the category of schemes that are weakly étale over $X$. It is equipped with a site structure where a family $(X_i\rar X)_{i \in I}$ is considered covering if, for every open affine $U$ of $X$, there exists an open affine of $\bigsqcup\limits_{i\in I} X_i$ that covers $U$.

The topological ring $R_\lambda$ defines, for any scheme $X$, a sheaf $R_{\lambda,X}$ on the proétale site $X_{\proet}$ of $X$ with values in the category $\D(\Z)$. This sheaf is an algebra in the category $\Shv(X_{\proet},\Z)$ of pro-étale hypersheaves on $X_{\proet}$.

\begin{definition}
\begin{enumerate}
    \item We define the presentably monoidal stable category $\mc{D}(X,R_\ell)$ of \emph{$\lambda$-adic sheaves} as the category of $R_{\lambda,X}$-modules in $\Shv(X_{\proet},\Z)$.
    \item A $\lambda$-adic sheaf is \emph{constructible} if there exists a stratification of $X$ such that the restriction of this sheaf to each stratum is dualizable. The full subcategory $\mc{D}^b_c(X,R_\lambda)$ of $\mc{D}(X,R_\lambda)$ formed by constructible objects is a monoidal subcategory.
\end{enumerate}
\end{definition}

Similarly, we can define the category of constructible étale sheaves, denoted by $\Shv_c(X_{\et},R)$.


To link those categories together, recall that we have a morphism of sites $\nu\colon S_{\proet}\rar S_{\et}$ which, for any integer $n$, induces a premotivic adjunction
$$\nu^*\colon \Shv(-_{\et},R/\lambda^nR)\leftrightarrows \Shv(-_{\proet},R/\lambda^n R)\colon \nu_*.$$
According to \cite[3.40]{hrs}, the latter is an equivalence and we even have the following result.
\begin{proposition}\label{constructibleBS}
Let $X$ be a scheme. The mapping $M\mapsto \lim \nu^*(M)/\ell^n$ induces an equivalence 
$$\Psi\colon \Shv_{c}(X_{\et},R_\lambda)\rar \mc{D}^b_c(X,R_\lambda).$$
\end{proposition}

\begin{definition}
The \emph{$\lambda$-adic realization} is the functor
$$\rho_\lambda=\Psi \rho^!\rho_\lambda^0\colon \DM_{\et,c}(S,R)\rar \mc{D}^b_c(S,R_\ell).$$
\end{definition}

We can derive the following result from \cite[7.2.16]{em}.

\begin{proposition}
The functor
$$\rho_\ell\colon \DM_{\et,c}(-,R)\rar \mc{D}^b_c(-,R_\ell)$$
commutes with the six functors when restricted to quasi-excellent schemes and morphisms of finite type between these schemes.
\end{proposition}

We finish this section by describing the rational version of the adic realization. Let $K$ be the fraction field of $R$. We denote by $K_\lambda$ the completion of $K$ for the $\lambda$-adic valuation. 

\begin{definition}
Let $\mc{C}$ be a stable $R$-linear category. We define the $K$-linearization $\mc{C}\otimes_R K$ of $\mc{C}$ as the idempotent completion of the stable $K$-linear category having the same objects as $\mc{C}$ and such that 
$$\Map_{\mc{C}\otimes_R K}(M,N)=\Map_{\mc{C}}(M,N)\otimes_R K.$$
\end{definition}

We define the category $\mc{D}^b_c(X,K_\ell)$ in the same fashion as $\mc{D}^b_c(X,R_\ell)$. According to \cite[7.2.23]{em}, we have
$$\mc{D}^b_c(S,K_\ell)=\mc{D}^b_c(S,R)\otimes_R K$$
and
$$\DM_{\et,c}(S,K)=\DM_{\et,c}(S,R)\otimes_R K.$$

We then define the $\lambda$-adic realization functor as follows.

\begin{definition}
The \emph{$\lambda$-adic realization} is the functor
$$\rho_\lambda\otimes_R K\colon \DM_{\et,c}(S,K)\rar \mc{D}^b_c(S,K_\lambda)$$
which is also denoted as $\rho_\lambda$.
\end{definition}

Since we have a pre-motivic equivalence between étale motives and h-motives, as shown in \cite[7.2.24]{em}, we obtain the following result (the reader is also referred to \cite[6.6]{ayo14}).

\begin{proposition}
The functor
$$\rho_\lambda\colon \DM_{\et,c}(-,K)\rar \mc{D}^b_c(-,K_\lambda)$$
commutes with the six operations when restricted to quasi-excellent schemes and morphisms of finite type between these schemes.
\end{proposition}

\subsubsection{The $v$-adic Realization}\label{l-adic real} Let $S$ be a scheme, let $R$ be a localization of the ring of integers of a number field $K$, let $v$ be a non-archemedian valuation on the field $K$ and let $\ell$ be the prime number such that $v$ extends the $\ell$-adic valuation.  Assuming that the prime number $\ell$ is invertible on $S$, we can define an $\ell$-adic realization functor, or in this context more accurately a $v$-adic realization functor $$\rho_v\colon\mc{DM}_{\et,c}(S,R)\rar \mc{D}^b_c(S,R_v)$$ which preserves constructible objects and respects the six functors formalism.
In the case when $v$ is non-negative on $R$, we can define the $v$-adic realization functor 

$$\rho_v\colon \DM_{\et}(S,R_{(v)})\rar \mc{D}(S,R_v)$$ by using the above construction since $R_{(v)}$ is a discrete valuation ring with finite residue field. We then define a $v$-adic realization functor by pre-composing with the change of coefficients from $R$ to $R_{(v)}$.

In the case when $v$ is not non-negative $R$, we can only define the realization functor on the subcategory of constructible motives. Notice that $R_v$ is the completion of $K$ with respect to the valuation $v$. Let $$A=\{ x\in K \mid v(x)\geqslant 0\}.$$ The ring $A$ is a discrete valuation ring. Let $\lambda$ be a uniformizer of $A$. We define the $v$-adic realization functor $\rho_v$ as the composition $$\DM_{\et,c}(S,R)\rar \DM_{\et,c}(S,K)\xrightarrow{\rho_\lambda \otimes_A K}\mc{D}^b_c(S,R_v),$$ 
where, with the notations of \Cref{p-loc}, \Cref{p-loc cons 1} ensures that the functor $$\DM_{\et,c}(S,A)_{(0)}\rar \DM_{\et,c}(S,K)$$ is an equivalence and where \cite[7.2.23]{em} ensures that the functor $$\mc{D}^b_c(S,A_v)_{(0)} \rar \mc{D}^b_c(S,K_v)$$ is an equivalence. 
\subsection{Artin Motives}
\subsubsection{Subcategories of Dimensional Motives}

Let $S$ be a scheme and $R$ be a ring. We can define subcategories of the stable category $\DM_{\et}(S,R)$ as follows:

\begin{itemize}
\item The category $\DM_{\et}^{n}(S,R)$ of \emph{$n$-étale motives} over $S$ is the localizing subcategory generated by the motives of the form $$h_S(X)=M_S^{\BM}(X)$$ where $X$ runs through the set of proper $S$-schemes of relative dimension at most $n$. 
\item The category $\DM_{\et,c}^{n}(S,R)$ of \emph{constructible $n$-étale motives} over $S$ is the thick subcategory generated by the motives of the form  $$h_S(X)=M_S^{\BM}(X)$$ where $X$ runs through the set of proper $S$-schemes of relative dimension at most $n$. 
\end{itemize}

The main focus of this text will be the following categories:

\begin{definition}\label{def AM} Let $S$ be a scheme and $R$ be a ring. We define
\begin{enumerate}
\item The category $\DM_{\et}^{A}(S,R)$ of \emph{Artin étale motives} to be the category of $0$-étale motives.
\item The category $\DM_{\et,c}^{A}(S,R)$) of \emph{constructible Artin étale motives} to be the category of constructible $0$-motives.
\item The category $\DM_{\et}^{smA}(S,R)$ of \emph{smooth Artin étale motives} over $S$ to be the localizing subcategory of $\DM_{\et}(S,R)$ generated by the motives of the form $$h_S(X)=M_S^{\BM}(X)=M_S(X)$$ where $X$ runs  through the set of étale covers over $S$.
\item The category  $\DM_{\et,c}^{smA}(S,R)$ \emph{constructible smooth Artin étale motives} over $S$ to be the thick subcategory of $\DM_{\et}(S,R)$ generated by the motives of the form $$h_S(X)=M_S^{\BM}(X)=M_S(X)$$ where $X$ runs  through the set of étale covers over $S$.
\end{enumerate}
\end{definition}

Using Zariski's Main Theorem \cite[18.12.13]{ega4}, (constructible) Artin étale motives are generated by the motives of the form $h_S(X)$ where $X$ runs through the set of finite $S$-schemes.

We can define subcategories of the stable category $\DA_{\et}^{\eff}(S,R)$ in a similar fashion and we will use similar notations.

Beware that in general (in the case of integral coefficients), all Artin motives that belong to the category $\DM_{\et,c}(S,R)$ of constructible étale motives are not constructible Artin motives.

The functors $\otimes$ and $f^*$ (where $f$ is any morphism) induce functors over the categories $\DM_{\et,(c)}^{(sm)A}(-,R)$. 

Finally, we have a continuity property for $n$-motives, Artin motives and smooth Artin motives.

\begin{proposition}\label{continuity} Let $\mr{D}_c$ be on of the fibered categories $\DM^{n}_{\et,c}$ or $\DM^{smA}_{\et,c}$. Let $R$ be a good ring. Consider a scheme $X$ which is the limit of a projective system of schemes with affine transition maps $(X_i)_{i\in I}$.

Then, in the $2$-category of stable monoidal categories, we have $$\colim \mr{D}_c(X_i,R) = \mr{D}_c(X,R).$$
\end{proposition}
\begin{proof} The proof is the same as \cite[6.3.9]{em}, replacing \cite[8.10.5, 17.7.8]{ega4-partie3} with \cite[1.24]{plh}.
\end{proof}

\subsubsection{Generators and Stability} We now state and generalize some results from \cite{az} and \cite{plh} that will be useful in our discussion. The following two results were stated in \cite{az} for schemes over a perfect field and extended over any schemes in \cite{plh} where they assume that $R=\Q$. However, their proofs apply if $R$ is any commutative ring.

\begin{proposition}(\cite[1.28]{plh})\label{AM.generators} Let $S$ be a scheme and let $R$ be a ring. The category $\mc{DM}^A_{\et}(S,R)$ (resp. $\DM^A_{\et,c}(S,R)$) is the localizing (resp. thick) subcategory of $\mc{DM}_{\et}(S,R)$ generated by any of the following families of objects.
\begin{enumerate}\item The family of the $M_S^{\BM}(X)$, with $X$ quasi-finite over $S$.
\item The family of the $M_S^{\BM}(X)=h_S(X)$, with $X$ finite over $S$.
\item The family of the $M_S^{\BM}(X)=M_S(X)$, with $X$ étale over $S$.
\end{enumerate}
\end{proposition}

\begin{proposition}\label{f_!}(\cite[1.17]{plh}) Let $R$ be a ring. The categories $\mc{DM}^A_{\et,(c)}(-,R)$ are closed under the following operations.
\begin{enumerate}\item The functor $f^*$, where $f$ is any morphism.
\item The functor $f_!$, where $f$ is a quasi-finite morphism.
\item Tensor product.
\end{enumerate}
Furthermore, the fibered category $\DM_{\et,(c)}(-,R)$ satisfies the localization property \eqref{AM.localization}.
\end{proposition}
If $j\colon U\rar S$ is an open immersion, the motive $j_*\un_U$ need not be Artin, the fibered category $\DM_{\et}(-,R)$ does not satisfy the localization property \eqref{AM.colocalization}.

On the other hand, the description of generators yields the following statement.
\begin{proposition}\label{small et site 2} Let $S$ be a scheme and let $R$ be a ring. Then, the functor $\rho_!$ of \Cref{small et site} has its essential image contained in $\DM^A_{\et}(S,R)$. We still denote by $$\rho_!\colon \Shv(S_{\et},R)\rar \DM^A_{\et}(S,R)$$ the induced functor.
\end{proposition}
\begin{proof}
    The functor $\rho_!$ of \Cref{small et site} is colimit preserving. Therefore, the result follows from \Cref{generatorze} and from the fact that if $X$ is an étale $S$-scheme, we have $$\rho_!R_S(X)=M_S(X).$$
\end{proof}

Finally, we can identify the $\pp$-localization of the category of constructible Artin motives 
\begin{proposition}\label{p-loc cons 3} Let $S$ be a scheme, let $R$ be a ring and let $\pp$ be a prime ideal of $\Z$. Then, with the notations of \Cref{p-loc}, the functor $$\DM^A_{\et,c}(S,R)_\pp\rar \DM^A_{\et,c}(S,R\otimes_\Z \Z_\pp)$$ is an equivalence.
\end{proposition}
\begin{proof} The proof is the same as the proof of \Cref{p-loc cons 1}.
\end{proof}

\subsubsection{Torsion Motives are Artin Motives} 
\begin{proposition}\label{torsion motives are Artin} Let $S$ be a scheme and let $R$ be a ring. The stable category of torsion étale motives $\mc{DM}_{\et,\mathrm{tors}}(S,R)$ is the subcategory of the stable category $\mc{DM}_{\et}^A(S,R)$ made of those Artin motives $M$ such that the motive $M\otimes_\Z \Q$ vanishes.
\end{proposition}
\begin{proof} If an Artin motive $M$ is such that the motive $M\otimes_\Z \Q$ vanishes, it is by definition a torsion motive. Therefore, it suffices to show that torsion motives are Artin motives. 

Let $M$ be a torsion motive. Using \Cref{rigidity V1}, we can decompose $M$ as $$M=\bigoplus\limits_{p \text{ prime}} M_p$$ where $M_p$ is of $p^\infty$-torsion. Therefore, we can assume that $M$ is of $p^\infty$-torsion for some prime number $p$. \Cref{rigidity V1} gives a $p^\infty$-torsion étale sheaf $N$ over the open subscheme $S[1/p]$ of $S$ such that $$M=(j_p)_!\rho_! N$$ where $j_p\colon S[1/p]\rar S$ is the open immersion. 

By \Cref{small et site 2}, the motive $\rho_!N$ is an Artin motive. Therefore, by \Cref{f_!}, the motive $M$ is an Artin motive.
\end{proof}

\section{Torsion \'Etale Motives and Torsion \'Etale Sheaves}
\subsection{The Six Functors and the Rigidity Theorem}
\begin{definition}
    Let $S$ be a scheme, let $R$ be a ring and let $p$ be a prime number.
    We define stable categories as follows.
\begin{enumerate}
    \item The stable category of \emph{torsion étale motives} $\mc{DM}_{\et,\mathrm{tors}}(S,R)$ as the subcategory of the stable category $\mc{DM}_{\et}(S,R)$ made of those objects $M$ such that the motive $M\otimes_\Z \Q$ vanishes.
    \item The stable category of \emph{$p^\infty$-torsion étale motives} $\mc{DM}_{\et,p^\infty-\mathrm{tors}}(S,R)$ as the subcategory of the stable category $\mc{DM}_{\et}(S,R)$ made of those objects $M$ such that the motive $M\otimes_\Z \Z[1/p]$ vanishes.
    \item The stable category of \emph{torsion étale sheaves} $\Shv_{\mathrm{tors}}(S_{\et},R)$ as the subcategory of the stable category $\Shv(S_{\et},R)$ made of those sheaves $M$ such that the sheaf $M\otimes_\Z \Q$ vanishes.
    \item The stable category of \emph{$p^\infty$-torsion étale sheaves} $\Shv_{p^\infty-\mathrm{tors}}(S_{\et},R)$ as the subcategory of the stable category $\Shv(S_{\et},R)$ made of those sheaves $M$ such that the sheaf  $M\otimes_\Z \Z[1/p]$ vanishes.
\end{enumerate}
\end{definition}

\begin{proposition}\label{6 fun torsion} Let $R$ be a ring and let $p$ be a prime number. Then, the six functors of $\DM_{\et}(-,R)$ induce functors over $\mc{DM}_{\et,\mathrm{tors}}(-,R)$ and $\mc{DM}_{\et,p^\infty-\mathrm{tors}}(-,R)$.
\end{proposition}
\begin{proof} The six functors commute with tensor product with $\Q$ and with $\Z[1/p]$ by \Cref{commutation} which implies the result.
\end{proof}

Scholze has shown in \cite[7.15]{6-fun-scholze} that the category $\Shv_{\mathrm{tors}}(-_{\et},\Z)$ is endowed with the six functors formalism. The proof applies to an arbitrary noetherian ring $R$. We deduce in the same fashion that they induce functors over $\Shv_{p^\infty-\mathrm{tors}}(-_{\et},R)$.

The following result is an analog of the rigidity theorem \cite[3.2]{bachmanrigidity} for torsion motives.

\begin{proposition}\label{rigidity V1} Let $S$ be a scheme and let $R$ be a good ring. Then, 
\begin{enumerate}
    \item The functor $$\prod_{p \text{ prime}} \mc{DM}_{\et,p^\infty-\mathrm{tors}}(S,R) \rar \mc{DM}_{\et,\mathrm{tors}}(S,R)$$
    which sends an object $(M_p)_{p \text{ prime}}$ to the object $\bigoplus\limits_{p \text{ prime}} M_p$ is an equivalence.
    \item Let $p$ be a prime number. Let $j_p\colon S[1/p]\rar S$ be the open immersion. Then, the functor $$j_p^*\colon \mc{DM}_{\et,p^\infty-\mathrm{tors}}(S,R) \rar \mc{DM}_{\et,p^\infty-\mathrm{tors}}(S[1/p],R)$$ is an equivalence.
    \item Let $p$ be a prime number. Let $S$ be a scheme over which $p$ is invertible. Then, the map $\rho_!$ of \Cref{small et site} induces an equivalence $$\Shv_{p^\infty-\mathrm{tors}}(S_{\et},R)\rar \mc{DM}_{\et,p^\infty-\mathrm{tors}}(S,R)$$
    \item All those equivalences commute with the six functors formalism.
\end{enumerate}
\end{proposition}
\begin{proof} \textbf{Proof of Assertion (1):}
For any prime number $p$, let $\Z(p^\infty)$ be Prüfer $p$-group. If $A$ is an abelian group, denote by $\underline{A}$ the image through the functor $$\rho_!\colon \Shv(S_{\et},R)\rar \DM_{\et}(S,R)$$ of the constant sheaf with value $A$. We have an exact triangle 
$$\un_S \rar \un_S \otimes_{\Z} \Z[1/p] \rar \underline{\Z(p^\infty)},$$

therefore, if $M$ is a $p^\infty$-torsion étale motive, we have $$M=M\otimes_\Z \underline{\Z(p^\infty)}[-1].$$

Furthermore, if $p$ and $q$ are distinct prime numbers, the canonical map $$\Z(q^\infty)\rar \Z(q^\infty) \otimes_\Z \Z[1/p]$$ is an isomorphism.

For any prime number $p$, let $M_p$ and $N_p$ be $p^\infty$-torsion étale motives. Notice that if $p$ is a prime number, we have $$\begin{aligned}\bigoplus\limits_{\substack{q \text{ prime} \\ q \neq p }} N_{q}&=\bigoplus\limits_{\substack{q \text{ prime} \\ q \neq p }} N_{q} \otimes_\Z \Z(q^\infty)[-1] \\
&=\bigoplus\limits_{\substack{q \text{ prime} \\ q \neq p }} N_{q} \otimes_\Z (\Z(q^\infty)\otimes_\Z \Z[1/p])[-1] \\
&=\left(\bigoplus\limits_{\substack{q \text{ prime} \\ q \neq p }} N_{q} \right) \otimes_\Z \Z[1/p]
\end{aligned}$$

Therefore, we get $$\Map\left(M_p,\bigoplus\limits_{\substack{q \text{ prime} \\ q \neq p }} N_{q}\right)=0.$$

Hence, the map
$$\prod_{p \text{ prime}} \Map(M_p,N_p)\rar \Map\left(\bigoplus\limits_{p \text{ prime}} M_p,\bigoplus\limits_{p \text{ prime}} N_p\right)$$
is an equivalence.  Therefore, the functor $$\prod_{p \text{ prime}} \mc{DM}_{\et,p^\infty-\mathrm{tors}}(S,R) \rar \mc{DM}_{\et,\mathrm{tors}}(S,R)$$ is fully faithful. 

We now prove that this functor is essentially surjective. We have an exact triangle $$\un_S \rar \un_S \otimes_\Z \Q\rar \underline{\Q/\Z}.$$  

Let now $M$ be a torsion étale motive. We get that $$M=M\otimes_\Z \underline{\Q/\Z}[-1].$$
Furthermore, we have 
 $$\und{\Q/\Z}=\bigoplus \limits_{p\text{ prime}} \und{\Z(p^\infty)}.$$
and since the functor $\rho_!$ commutes with colimits, we therefore get that $$M=\bigoplus \limits_{p\text{ prime}} \left(M\otimes_\Z \und{\Z(p^\infty)}[-1]\right).$$
As $$\left(M\otimes_\Z \und{\Z(p^\infty)}[-1]\right)\otimes_\Z \Z[1/p]=0,$$
the motive $M\otimes_\Z \und{\Z(p^\infty)}[-1]$ is of $p^\infty$-torsion and therefore, the functor $$\prod_{p \text{ prime}} \mc{DM}_{\et,p^\infty-\mathrm{tors}}(S,R) \rar \mc{DM}_{\et,\mathrm{tors}}(S,R)$$ is essentially surjective which finishes the proof of the first assertion.

\textbf{Proof of Assertion (2):}

The $2$-morphism $$j_p^* (j_p)_* \rar \id$$ is an equivalence. Since by \Cref{6 fun torsion}, both functors $j_p^*$ and $(j_p)_*$ induce functors on $p^\infty$-torsion motives, the functor $$j_p^*\colon \mc{DM}_{\et,p^\infty-\mathrm{tors}}(S,R) \rar \mc{DM}_{\et,p^\infty-\mathrm{tors}}(S[1/p],R)$$ is essentially surjective. 

To show that the functor $j_p^*$ is fully faithful when restricted to the category of $p^\infty$-torsion étale motives, it suffices to show that the natural map $\id \rar (j_p)_* j_p^*$ is an equivalence when restricted to the category of $p^\infty$-torsion étale motives. Notice that the functor:
$$-\otimes_\Z \Z/p\Z \colon \DM_{\et,p^\infty-\mathrm{tors}}(S,R)\rar \DM_{\et}(S,R/pR)$$ is conservative by \Cref{conservativity} and commutes with the functors $(j_p)_*$ and $j_p^*$ by \cite[5.4.5]{em}. Therefore the full faithfulness of $j_p^*$ follows from the same result with coefficients in $R/pR$.

Since $R/pR$ is a $\Z/p\Z$-algebra, we have by \cite[5.4]{em} a premotivic adjunction $$\DM_{\et}(S,\Z/p\Z)\leftrightarrows \DM_{\et}(S,R/pR).$$
Therefore, the result follows from \cite[A.3.4]{em} 

\textbf{Proof of Assertion (3):}

Let $N$ be a $p^\infty$-torsion étale motive over $S$. Then, as before, we have $$N=N\otimes_\Z \underline{\Z(p^\infty)}[-1]=\colim N \otimes_\Z \Z/p^n\Z[-1].$$

If $n$ is a non-negative integer, let 
$$N_n=N \otimes_\Z \Z/p^n\Z[-1].$$ 
The étale motive $N_n$ lies in $\DM_{\et}(S,R/p^nR)$ and lies therefore in the essential image of the functor $\rho_!$ by the rigidity theorem \cite[3.2]{bachmanrigidity}. 
Let $M_n$ be a $p^\infty$-torsion étale sheaf such that 
$$\rho_!M_n=N_n.$$ 
The rigidity theorem also ensures that the map $$\id \otimes_\Z (\times p)\colon N_{n}\rar N_{n+1}$$
is the image through $\rho_!$ of a map $$\phi_n\colon M_{n}\rar M_{n+1}.$$ 
Letting $M$ be the colimit of the $M_n$ with transition maps the $\phi_n$, we have $$\rho_!M=N$$ since the functor $\rho_!$ preserves colimits.

Hence, the functor $\rho_!$ is essentially surjective.

Recall now that the map $\rho_!$ has a right adjoint functor $\rho^!$ defined in \Cref{small et site}. To finish the proof it suffices to show that the natural map $$\id \rar \rho^!\rho_!$$ is an equivalence. Since the functor $-\otimes_\Z \Z/p\Z$ is conservative and commutes with the functors $\rho^!$ and $\rho_!$ by \Cref{commutation}, it suffices to show the same result with coefficients $\Z/p\Z$ which follows from the rigidity theorem \cite[3.2]{bachmanrigidity}.

\textbf{Proof of Assertion (4):}

The fourth assertion follows from the fact that all those functors can be promoted into $\Sm$-premotivic functors (as defined in \cite[A.1.7]{em}) and from \cite[A.1.17]{em}.
\end{proof}

Putting those results together, we get the following result.
\begin{corollary}\label{rigidity V2} Keep the notations of \Cref{rigidity V1}. The functor $$\prod_{p \text{ prime}} \Shv_{p^\infty-\mathrm{tors}}(S[1/p]_{\et},R) \rar \mc{DM}_{\et,\mathrm{tors}}(S,R)$$
which sends an object $(M_p)_{p \text{ prime}}$ to the object $\bigoplus\limits_{p \text{ prime}} j_p^*(\rho_! M_p)$ is an equivalence. This equivalence commutes with the six functors.
\end{corollary}

\subsection{The Ordinary t-structure}

\begin{proposition}\label{generators ord} Let $S$ be a scheme and let $R$ be a ring. The ordinary t-structure on $\Shv(S_{\et},R)$ is generated in the sense of \Cref{AM.t-structure generated} by the étale sheaves of the form $R_S(X)$ with $X$ étale over $S$.
\end{proposition}
\begin{proof} 
The proposition follows from \Cref{generatorze}.
\end{proof}

\begin{proposition}\label{ordinary torsion sheaves}
Let $S$ be a scheme and $R$ be a ring. Then, the ordinary t-structure of the stable category $\Shv(S_{\et},R)$ induces a t-structure on the stable subcategory $\Shv_{p^\infty-\mathrm{tors}}(S_{\et},R)$. 
\end{proposition} 
\begin{proof} Let $\tau_{\geqslant 0}$ be the truncation with respect to the ordinary t-structure of $\Shv(S_{\et},R)$. Let $M$ be an étale sheaf of $p^\infty$-torsion. The functor $\otimes_\Z \Z[1/p]$ is t-exact with respect to the ordinary t-structure as $\Z[1/p]$ is flat over $\Z$. Therefore, we have $$\tau_{\geqslant 0}(M)\otimes_\Z \Z[1/p]=\tau_{\geqslant 0}(M\otimes_\Z \Z[1/p]).$$

Furthermore, the sheaf $\tau_{\geqslant 0}(M\otimes_\Z \Z[1/p])$ vanishes and therefore, the sheaf $\tau_{\geqslant 0}(M)$ is of $p^\infty$-torsion. We conclude by using \cite[1.3.19]{bbd}.
\end{proof}

A product of t-categories is endowed with a t-structure by taking the product of the subcategories of t-non-positive objects as the subcategory of t-non-positive objects and the product of the subcategories of t-non-negative objects as the subcategory of t-non-negative objects.

This allows us to define an ordinary t-structure on $p^\infty$-torsion étale motives.

\begin{definition}\label{ordinary torsion} Let $S$ be a scheme and let $R$ be a good ring. The \emph{ordinary t-structure} over the stable category $\mc{DM}_{\et,\mathrm{tors}}(S,R)$ is the only t-structure such that the functor  $$\prod_{p \text{ prime}} \Shv_{p^\infty-\mathrm{tors}}(S[1/p]_{\et},R) \rar \mc{DM}_{\et,\mathrm{tors}}(S,R)$$
of \Cref{rigidity V2} is a t-exact equivalence when the left hand side is endowed with the product of the ordinary t-structures.
\end{definition}

\begin{remark} Let $S$ be a scheme and let $R$ be a good ring. If all the residue characteristics of $S$ are invertible in $R$, then, the functor $\rho_!$ of \Cref{small et site} induces a t-exact equivalence $$\Shv_{\mathrm{tors}}(S_{\et},R) \rar \mc{DM}_{\et,\mathrm{tors}}(S,R)$$ when both sides are endowed with their ordinary t-structures.    
\end{remark}

\begin{proposition}\label{f^* ord torsion}
    Let $f\colon T\rar S$ be a morphism of schemes and let $R$ be a good ring. The functor $f^*$ induces a t-exact functor $$\mc{DM}_{\et,\mathrm{tors}}(S,R)\rar\mc{DM}_{\et,\mathrm{tors}}(T,R)$$ when both sides are endowed with their ordinary t-structures.
\end{proposition}
\begin{proof}
    Since the functor of \Cref{rigidity V2} commutes with the six functors, it suffices to show that if $p$ is a prime number and is invertible on $S$, the functor $$f^*\colon \Shv_{p^\infty-\mathrm{tors}}(S_{\et},R)\rar \Shv_{p^\infty-\mathrm{tors}}(T_{\et},R)$$ is t-exact when both sides are endowed with their ordinary t-structures which follows from the t-exactness of the functor $$f^*\colon \Shv(S_{\et},R)\rar \Shv(T_{\et},R).$$
\end{proof}

\section{Description of Smooth Artin Motives and Applications}\label{section 2}
\subsection{Artin Representations and Lisse \'Etale Sheaves}

\subsubsection{Discrete representations, Artin representations}

 Recall the following definitions:
\begin{definition} Let $\pi$ be a topological group and let $R$ be a ring. An $R[\pi]$-module $M$ is \emph{discrete} if the action of $\pi$ on $M$ is continuous when $M$ is endowed with the discrete topology. We will denote by $\mathrm{Mod}(\pi,R)$ the Grothendieck abelian category of discrete $R[\pi]$-modules.
\end{definition}
\begin{remark} Keep the same notations. An $R[\pi]$-module is discrete if and only if the stabilizer of every point is open. Equivalently, this means that $M$ is the union of all the $M^U$ for $U$ an open subgroup of $\pi$.

If the topological group $\pi$ is profinite, this also means that $M$ is the union of all the $M^H$ for $H$ a normal subgroup of $\pi$ of finite index.
\end{remark}

\begin{definition} Let $\pi$ be a profinite group and let $R$ be a ring. An \emph{Artin representation} of $\pi$  is a discrete $R[\pi]$-module which is of finite presentation as an $R$-module. We will denote by $\Rep^A(\pi,R)$ the abelian category of Artin representations.
\end{definition}

\begin{remark} Keep the same notations. An $R[\pi]$-module is an Artin representation if and only if it is of finite presentation as an $R$-module and the action of $\pi$ factors through a finite quotient. Indeed, if $M$ is an Artin representation, the stabilizer of a finite generating family is open and thus contains a normal subgroup $N$ of finite index. Therefore, the action of $\pi$ factors through $\pi/N$.
\end{remark}

\subsubsection{Locally constant étale sheaves}

\begin{definition}\label{shv lis} Let $S$ be a scheme and let $R$ be a regular ring. We define:
	\begin{enumerate}
        \item The abelian category $\Sh(S_{\et},R)$ of \emph{étale sheaves of $R$-modules} as the heart of the ordinary t-structure on the stable category $\Shv(S_{\et},R)$.
		\item The abelian category $\mathrm{Loc}_S(R)$ of \emph{locally constant sheaves of $R$-modules} as the subcategory of the abelian category $\Sh(S_{\et},R)$ made of those sheaves which are étale-locally constant with finitely presented fibers.
        \item The abelian category $\In\mathrm{Loc}(S,R)$ of \emph{Ind-locally constant sheaves of $R$-modules} as the subcategory of the abelian category $\Sh(S_{\et},R)$ made of those sheaves which are filtered colimits of locally constant sheaves of $R$-modules.
		\item The \emph{category of lisse étale sheaves} $\Shv_{\mathrm{lisse}}(S,R)$ as the subcategory of dualizable objects of the stable category $\Shv(S_{\et},R)$.
        \item The \emph{category of Ind-lisse étale sheaves} $\Shv_{\In\mathrm{lisse}}(S,R)$ as the localizing subcategory of the stable category $\Shv(S_{\et},R)$ generated by $\Shv_{\mathrm{lisse}}(S,R)$.
	\end{enumerate}
\end{definition}

These categories are related to each other by the following proposition.
\begin{proposition}\label{loc cst} Let $S$ be a scheme and let $R$ be a regular ring. 
\begin{enumerate}
    \item An étale sheaf is lisse if and only if it is locally constant with perfect fibers.
    \item The ordinary t-structure on the stable category $\Shv(S_{\et},R)$ induces a t-structure on the subcategory $\Shv_{\mathrm{lisse}}(S,R)$ whose heart is the abelian category $\mathrm{Loc}_S(R)$.
    \item The ordinary t-structure on the stable category $\Shv(S_{\et},R)$ induces a t-structure on the subcategory $\Shv_{\In\mathrm{lisse}}(S,R)$ whose heart is the abelian category $\In\mathrm{Loc}(S,R)$.
\end{enumerate}
\begin{proof}
    The first assertion follows from \cite[0FPP]{stacks}. To prove the second assertion, notice that, as a consequence of the first assertion, lisse sheaves are bounded with respect to the ordinary t-structure of $\Shv(S_{\et},R)$. Therefore, \cite[Lemma 1.2.3]{aps} ensures that it suffices to show that $$\Loc_S(R)=\Shv_{\lis}(S,R)\cap \Shv(S_{\et},R)^{\heart}.$$

The fact that locally constant sheaves with finitely presented fibers have perfect fibers follows from the hypothesis that the ring $R$ is regular, since in this case, an $R$-modules is perfect if and only if it is finitely presented by \cite[066Z]{stacks}. Moreover, an abelian sheaf with perfect fibers is finitely presented. This finishes the proof of the second assertion.

Any Ind-lisse étale sheaf is a filtered colimit of lisse étale sheaves. Furthermore, by \cite[1.3.5.21]{ha}, the ordinary t-structure of $\Shv(S_{\et},R)$ is compatible with filtered colimits. Therefore, if $M=\colim_{i\in I} M_i$ is an Ind-lisse étale sheaf, where $I$ is a filtered simplicial set and each $M_i$ is lisse, we have $$\tau_{\geqslant 0}(M)=\tau_{\geqslant 0}(\colim M_i)=\colim \tau_{\geqslant 0}(M_i).$$

Therefore, the sheaf $\tau_{\geqslant 0}(M)$ is Ind-lisse. Hence, by \cite[1.3.19]{bbd}, the ordinary t-structure of $\Shv(S_{\et},R)$ induces a t-structure on the category $\Shv_{\In\mathrm{lisse}}(S,R)$. An object $L$ of the heart of this t-structure is a filtered colimit of lisse sheaves $M_i$. But then, we have, $$L=\colim M_i=H^0(\colim M_i)=\colim H^0(M_i).$$

Therefore, the sheaf $L$ is Ind-locally constant. Conversely, Ind-locally constant sheaves lie in the heart of the ordinary t-structure and are Ind-lisse. This finishes the proof of the third assertion.
\end{proof}

\end{proposition}

The above proposition shows in particular that $\Loc_S(R)$ is a weak Serre subcategory of $\Sh(S_{\et},R)$. We have the following characterization of lisse sheaves:

\begin{proposition}\label{description des faisceaux lisses} Let $S$ be a connected scheme and let $R$ be a regular ring. Let $\xi$ be a geometric point of $S$. The fiber functor associated to $\xi$ induces an equivalence of abelian monoidal categories $$\xi^*\colon\mathrm{Loc}_S(R)\rar \Rep^A(\pi_1^{\et}(S,\xi),R)$$
where $\pi_1^{\mathrm{\et}}(S,\xi)$ is the étale fundamental group of $S$ with base point $\xi$ of \cite[V]{sga1}.
\end{proposition}
\begin{proof} Let $\mathrm{Loc}_S$ be the Galois category of locally constant sheaves of sets on $S_{\et}$. By definition of the étale fundamental group, the fiber functor $$\xi^*\colon\mathrm{Loc}_S \rar \pi_1^{\et}(S,\xi)-\mathrm{Set}$$ is an equivalence.

The constant sheaf $R_{S}(S)$  is sent to the set $R$ endowed with the trivial action of $\pi_1^{\et}(S,\xi)$. Thus, the functor $\xi^*$ induces an equivalence $$\mathrm{Loc}_S(R)\rar \Rep^A(\pi_1^{\et}(S,\xi),R).$$
\end{proof}

In the case when $S$ is the spectrum of a field $k$, the results of this paragraph can be formulated in a more satisfying way.

 \begin{proposition}\label{lisse fields} Let $k$ be a field and let $R$ be a regular ring. 
 \begin{enumerate}
     
    \item If $\xi\colon\Spec(\overline{k})\rar \Spec(k)$ is the separable closure of $k$, the fiber functor associated to $\xi$ induces an equivalence of abelian monoidal categories $$\xi^*\colon\Sh(k_{\et},R)\rar \mathrm{Mod}(G_k,R)$$
and an equivalence of monoidal stable categories
$$\xi^*\colon\Shv(k_{\et},R)\rar \mc{D}(\mathrm{Mod}(G_k,R)).$$
    \item The functor $\xi^*$ induces an equivalence $$\xi^*\colon\Shv_{\lis}(k_{\et},R)\rar \mc{D}_{\mathrm{perf}}(\mathrm{Mod}(G_k,R))$$ where $\mc{D}_{\mathrm{perf}}(\mathrm{Mod}(G_k,R))$ denotes the subcategory of the stable category $\mc{D}(\mathrm{Mod}(G_k,R))$ made of those complexes which are perfect as complexes of $R$-modules.
    \item We have 
$$\Shv(k_{\et},R)=\Shv_{\In\lis}(k_{\et},R)$$
$$\Sh(k_{\et},R)=\In\Loc(k,R).$$
    \item The canonical functor $$\mc{D}^b(\Rep^A(G_k,R))\rar \mc{D}_{\mathrm{perf}}(\mathrm{Mod}(G_k,R))$$ is an equivalence.
   \end{enumerate}
 \end{proposition} 
\begin{proof} The first assertion follows from \cite[VIII.2.1]{sga4}.

The second assertion follows from the first assertion and the description of lisse sheaves \Cref{loc cst}(1).

Furthermore, any complex of discrete $R$-linear $G_k$-module is a (filtered) colimit of complexes with perfect underlying complex of $R$-modules. Therefore, we get $$\Shv(k_{\et},R)=\Shv_{\In\lis}(k_{\et},R)$$ and since the ordinary t-structure on the stable category $\Shv(k_{\et},R)$ is compatible with filtered colimits, we also get $$\Sh(k_{\et},R)=\In\Loc(k,R).$$

The last assertion is a consequence of \cite[I.4.12]{sga6}. We give an alternative proof for the reader's convenience. Denote by $$\psi\colon \mc{D}^b(\Rep^A(G_k,R))\rar \mc{D}_{\mathrm{perf}}(\mathrm{Mod}(G_k,R))$$ the canonical functor.
It suffices to show that the canonical map $$\Map_{\mc{D}^b(\Rep^A(G_k,R))}(M,N)\rar \Map_{\mc{D}(\mathrm{Mod}(G_k,R)))}(\psi(M),\psi(N))$$ is an equivalence for any bounded complexes $M$ and $N$ of Artin representations of $G_k$.

Notice that the objects of $\mc{D}^b(\Rep^A(G_k,R))$ are dualizable. Therefore, we can assume that $M$ is the $R$-module $R$ endowed with the trivial $G_k$-action placed in degree $0$. Furthermore, by dévissage, we can assume that $N$ is placed in degree $0$.

Therefore, it suffices to show that for all $n$, the map $$\mathrm{Ext}^n_{\Rep^A(G_k,R)}(R,N)\rar H^n(G_k,N)$$ is an equivalence.

Now, we have  $$H^n(G_k,N)=\colim H^n(G_k/H,N^H)$$ where $H$ runs through the set of normal subgroups of finite index of $G_k$. 

Since $N$ is an Artin representation, any normal subgroup $H$ of $G_k$ contains a normal subgroup $H'$ such that the action of $H'$ on $N$ is trivial. Hence, the set of normal subgroups $H$ of finite index of $G_k$ such that $N^H=N$ is final in the set of all normal subgroups of finite index of $G_k$. 

Therefore, we have $H^n(G_k,N)=\colim H^n(G_k/H,N)$ where $H$ runs through the set of normal subgroups of finite index of $G_k$ such that the action of $H$ on $N$ is trivial.

In addition, if $H$ is a normal subgroup of $G_k$ which acts trivially on $N$, the abelian group $H^n(G_k/H,N)$ coincides with the $\mathrm{Ext}^n_{\Rep^A(G_k/H,R)}(R,N)$ and can by \cite[06XU]{stacks} be computed as the abelian group of classes of sequences $$0\rar R\rar M_1\rar \cdots \rar M_n \rar N\rar 0$$
where each $M_i$ is an $R$-linear representation of $G_k/H$ and two sequences $M$ and $M'$ are equivalent if there exists a third sequence $M''$ and morphisms of sequences $M\leftarrow M''\rar M'$.

Furthermore, using \cite[06XU]{stacks} again, the abelian group $\mathrm{Ext}^n_{\Rep^A(G_k,R)}(R,N)$ can be computed as the abelian group of classes of sequences $$0\rar R\rar M_1\rar \cdots \rar M_n \rar N\rar 0$$
where each $M_i$ is an Artin representation of $G_k$ and two sequences $M$ and $M'$ are equivalent if there exists a third sequence $M''$ and morphisms of sequences $M\leftarrow M''\rar M'$.

Any such sequence in fact is a sequence of $R$-linear representation of $G_k/H$ for some normal subgroup $H$ and any equivalence $M\leftarrow M''\rar M'$ between two sequences $M$ and $M'$ comes from an equivalence of sequences of $R$-linear representation of $G_k/H$ for some normal subgroup $H$. 
Therefore, the map $$\colim\mathrm{Ext}^n_{\Rep^A(G_k/H,R)}(R,N)\rar\mathrm{Ext}^n_{\Rep^A(G_k,R)}(R,N),$$ where $H$ runs through the set of normal subgroups of finite index of $G_k$ such that the action of $H$ on $N$ is trivial, is an equivalence.
This finishes the proof.
\end{proof}

\subsubsection{Generators of lisse and Ind-lisse sheaves}

\begin{proposition}\label{smooth generators}
    Let $S$ be a scheme and let $R$ be a regular ring. 
    \begin{enumerate}
        \item The thick stable subcategory of $\Shv(S_{\et},R)$ generated by the sheaves of the form $R_S(X)$ with $X$ finite and étale over $S$ is $\Shv_{\lis}(S,R)$.
        \item The localizing stable subcategory of $\Shv(S_{\et},R)$ generated by the sheaves of the form $R_S(X)$ with $X$ finite and étale is $\Shv_{\In\lis}(S,R)$.
    \end{enumerate}
\end{proposition}
\begin{proof} By definition of Ind-lisse sheaves the second assertion follows from the first. 

To prove the first assertion, first note that $\Shv_{\lis}(S,R)$ is a thick stable subcategory of $\Shv(S_{\et},R)$. Furthermore, if $X$ is a finite étale $S$-scheme, the sheaf $R_S(X)$ is dualizable with dual itself and is therefore lisse. Hence, the thick stable subcategory of $\Shv(S_{\et},R)$ generated by the sheaves of the form $R_S(X)$ with $X$ finite and étale over $S$ is contained in $\Shv_{\lis}(S,R)$.

By \Cref{loc cst}, lisse sheaves are bounded with respect to the ordinary t-structure. Therefore, it suffices by dévissage to prove that any locally constant sheaf lies in the thick subcategory generated by the sheaves of the form $R_S(X)$ with $X$ finite and étale over $S$ which follows from the stronger \Cref{permutation resolution 2} below. 
\end{proof}
\begin{lemma}\label{permutation resolution 2} Let $S$ be a scheme and let $R$ be a regular ring and let $\mc{C}(S,R)$ be the smallest subcategory of $\Shv(S_{\et},R)$ which contains the sheaves of the form $R_S(X)$ with $X$ finite and étale over $S$ and is closed under finite colimits, extensions and retracts. Then, the category $\Loc_S(R)$ is contained in $\mc{C}(S,R)$.
\end{lemma}
\begin{proof}
We can assume that the scheme $S$ is connected. Let $\xi$ be a geometric point of $S$ and let $M$ be a locally constant étale sheaf. By \Cref{description des faisceaux lisses}, the sheaf $M$ corresponds to an Artin representation $N$ of $\pi_1^{\et}(S,\xi)$ through the fiber functor associated to $\xi$.
Let $G$ be a finite quotient of $\pi_1^{\et}(S,\xi)$ through which $N$ factors. Then, by \cite[1.3.1]{aps} (which is a result of Balmer and Gallauer) the complex $M \oplus M[1]$ of finite $R[G]$-modules is quasi-isomorphic to a complex of the form $$P=\cdots \rar 0 \rar P_n \rar P_{n-1} \rar \cdots \rar P_0 \rar 0 \rar \cdots$$ where each $P_i$ is placed in degree $-i$ and is isomorphic to $R[G/H_1]\oplus \cdots \oplus R[G/H_m]$ for some subgroups $H_1,\ldots,H_m$ of $G$.
	
Furthermore, we have an exact functor 
$$\mc{D}^b(\Rep(G,R))\rar \Shv(S_{\et},R)$$
and if $H$ is a subgroup of $G$, the finite set $G/H$ is endowed with an action of $\pi_1^{\et}(S,\xi)$ and therefore corresponds to a finite locally constant sheaf on $S_{\et}$, and therefore to a finite étale $S$-scheme $X$. The image of $R[G/H]$ through the functor above is then $R_S(X)$.
	
Thus, the object $M\oplus M[1]$, and therefore the object $M$ itself, belong to the smallest subcategory which contains the sheaves of the form $R_S(X)$ with $X$ finite and étale over $S$ and is closed under finite colimits, extensions and retracts.
\end{proof}

Another consequence of \Cref{permutation resolution 2} is the following statement. 
\begin{proposition}\label{smooth generators t-structure}
    Let $S$ be a scheme and let $R$ be a regular ring. 
    \begin{enumerate}
        \item The smallest subcategory of $\Shv_{\lis}(S,R)$ which contains the sheaves of the form $R_S(X)$ with $X$ finite and étale over $S$ and is closed under finite colimits, extensions and retracts is the subcategory $\Shv_{\lis}^{\leqslant 0}(S,R)$ of t-non-positive objects of the ordinary t-structure on $\Shv_{\lis}(S,R)$.
        \item The ordinary t-structure on the stable category $\Shv_{\In\lis}(S,R)$ is generated by the sheaves of the form $R_S(X)$ with $X$ finite and étale over $S$.
    \end{enumerate}
\end{proposition}
\begin{proof} By dévissage, the first assertion follows from \Cref{permutation resolution 2}.

Let $M$ be a t-non-positive Ind-lisse étale sheaf. Write $M$ as a filtered colimit of lisse étale sheaves $M_i$. Since the ordinary t-structure is compatible with filtered colimits, we have $$M=\tau_{\leqslant 0}(M)=\tau_{\leqslant 0}(\colim M_i)=\colim\tau_{\leqslant 0}(M_i).$$

Using the first assertion, each sheaf $\tau_{\leqslant 0}(M_i)$ belongs to the smallest subcategory of the stable category $\Shv_{\In\lis}(S,R)$ which contains the sheaves of the form $R_S(X)$ with $X$ finite and étale over $S$ and is closed under small colimits and extensions. Therefore, so does the sheaf $M$. This finishes the proof. 
\end{proof}

\begin{proposition}\label{p-loc cons 2} Let $S$ be a scheme and let $R$ be a regular ring. Then, with the notations of \Cref{p-loc}, the canonical functor $$\Shv_{\lis}(S,R)_\pp\rar \Shv_{\lis}(S,R\otimes_\Z \Z_\pp)$$ is an equivalence.   
\end{proposition} 
\begin{proof} The proof is the same as the proof of \Cref{p-loc cons 1}.
\end{proof}

\subsection{Computations of \'Etale Cohomology and \'Etale Motivic Cohomology}

Recall that the étale motivic cohomology groups of a scheme $X$ are the groups 
$$\Hl^{n,m}_{\mc{M},\et}(X,R)=\Hom_{\DMtri_{\et}(X,R)}(\un_X,\un_X[n](m)).$$

We will need the following computations. The first one is a slight generalization of a \cite[11.1]{ayo14}.

\begin{lemma}\label{n<0 m<=0}Let $S$ be a scheme, let $R$ be a good ring and let $n<0$ and $m\leqslant 0$ be integers. Then, $$\Hl^{n,m}_{\mc{M},\et}(S,R)=0.$$
\end{lemma}
\begin{proof} By \Cref{conservativity}, we are reduced to the case where $R$ is a $\Z/p\Z$-algebra or a $\Q$-algebra. 

In the first case, by the rigidity theorem \cite[3.2]{bachmanrigidity}, the category $\DM_{\et}(S,R)$ is equivalent to the category $\mc{D}(S[1/p]_{\et},R)$ and we have $$\Hl^{n,m}_{\mc{M},\et}(X,R)=\Hl^n_{\et}(X,R(m))$$ which vanishes since the integer $n$ is negative. In the case when the ring $R$ is a $\Q$-algebra, Ayoub's result \cite[11.1]{ayo14} applies.
\end{proof}

We will also need to compute étale cohomology with coefficients invertible on the base scheme and with constant coefficients. First recall:
\begin{proposition}\label{den}(Deninger) Let $X$ be a geometrically unibranch scheme which is the disjoint union of its irreducible components. Then, we have $$\Hl^i_{\et}(X,\Q)=\begin{cases} \Q^{\pi_0(X)} \text{ if }i=0. \\
0 \text{ otherwise.}\end{cases}$$
\end{proposition}
\begin{proof}By \cite[2.1]{deninger}, this result is true when the scheme $X$ is normal. If $X$ is a geometrically unibranch scheme which is the disjoint union of its irreducible components, the normalization map $\nu\colon X^\nu \rar X$ is a universal homeomorphism by \cite[0GIQ]{stacks} and therefore, the lemma follows from the topological invariance of the small étale site \cite[18.1.2]{ega4}.
\end{proof}

If $X$ is a scheme and $\ell$ is a prime number which is invertible on $X$, write $b_n(X,\ell)$ the dimension (over $\Q_\ell$) of the $n$-th $\ell$-adic cohomology group of $X$. Let $T_n(X,\ell)$ be the torsion subgroup of the group $\Hl^n_{\et}(X,\Z_\ell)$. Finally denote by $\Z(\ell^\infty)$ the Prüfer $\ell$-group. Then we have the following result. 
\begin{proposition}\label{h-cohomologie motivique sur un corps} Let $X$ be a geometrically unibranch scheme which is the disjoint union of its irreducible components, let $R$ be a ring which is flat over $\Z$ and in which all the residue characteristics of $X$ are invertible and let $e(X)$ be the set of residue characteristic exponents of $X$. Then, $$\Hl^{n}_{\et}(X,R)=\begin{cases}R^{\pi_0(X)} &\text{ if }n=0\\
0 &\text{ if }n=1\\
R \otimes_\Z \left[\bigoplus\limits_{\ell\notin e(X)}\left(\Z(\ell^\infty)^{b_{n-1}(X,\ell)}\oplus T_{n-1}(X,\ell)\right)\right]  &\text{ if }n\geqslant 2
\end{cases}$$
\end{proposition}
\begin{proof}As $R$ is flat over $\Z$, we may suppose that $R$ is the subring $$\Z[1/e(X)]=\left\{\frac{n}{ p_1^{m_1}\cdots p_r^{m_r}}\mid p_i \in e(X)\right\}$$ of $\Q$. We have an exact sequence 

$$0\rar \Z[1/e(X)]\rar \Q\rar \bigoplus\limits_{\ell\notin e(X)}\Z(\ell^\infty)\rar 0.$$

Hence, by \Cref{den}, we get $$\Hl^n_{\et}(X,\Z[1/e(X)])=\begin{cases}\Z[1/e(X)]^{\pi_0(X)} &\text{ if }n=0\\
0 &\text{ if }n=1\\
\bigoplus\limits_{\ell\notin e(X)}\colim_N \Hl^{n-1}_{\et}(X,\Z/\ell^N\Z) &\text{ if }n\geqslant 2
\end{cases}$$
This yields the result.
\end{proof}
\begin{remark}\label{etale coh of curves} If a complete theory of abelian motives was available, all the $b_i(X,\ell)$ would coincide. Thus, for $n\geqslant 2$, we would have  $$\Hl^{n}_{\et}(X,R)=R\otimes_\Z \left[\Q/\Z[1/e(X)]^{b_i(X)}
\oplus\bigoplus\limits_{\ell\notin e(X)} T_{n-1}(X,\ell) \right]$$

Notice that, using Artin's Theorem \cite[XI.4.4]{sga4} or Deligne's work on Weil's conjectures (see \cite{WeilII}) all the $b_i(X,\ell)$ coincide if $X$ is defined over an algebraically closed field of characteristic $0$ or over $\overline{\mb{F}}_p$.

Finally, if $X$ is a curve over any separably closed field $k$, using \cite[IX.4.7]{sga4}, we have $$\Hl^{2}_{\et}(X,R)=(R\otimes_\Z \Q/\Z[1/p])^{2g} $$ 
$$\Hl^{3}_{\et}(X,R)=R\otimes_\Z \Q/\Z[1/p] $$ 

and $\Hl^{n}_{\et}(X,R)=0$ for $n \geqslant 4$.
\end{remark}
\subsection{Equivalence Between Smooth Artin Motives and Lisse \'Etale Sheaves over a Regular Base scheme}
In this paragraph, we give an explicit description of smooth Artin étale motives over a regular base scheme. This description will not work over a non-regular base scheme, even when it is normal (see \Cref{B7} below). 

\begin{theorem}\label{smooth Artin h-motives} Let $R$ be a regular good ring and let $S$ be a regular scheme. Assume that the residue characteristic exponents of $S$ are invertible in $R$.

The functor $\rho_!$ of \Cref{small et site 2} induces monoidal equivalences
$$\Shv_{\In\lis}(S,R)\longrightarrow \DM^{smA}_{\et}(S,R),$$
$$\Shv_{\lis}(S,R)\longrightarrow \DM^{smA}_{\et,c}(S,R).$$
\end{theorem}
\begin{proof} We can assume the scheme $S$ to be connected. If $V$ is an étale cover of $S$, the functor $\rho_!$ sends the sheaf $R_S(V)$ to the motive $M_S(V)$. Using \Cref{smooth generators}, the fact that the ring $R$ is regular and the fact that $\rho_!$ commutes with colimits and finite limits, it induces exact functors 
$$\Shv_{\In\lis}(S,R)\longrightarrow \DM^{smA}_{\et}(S,R)$$
$$\Shv_{\lis}(S,R)\longrightarrow \DM^{smA}_{\et,c}(S,R).$$

Assuming that the second functor is an equivalence, the first functor is also an equivalence since $\Shv_{\lis}(S,R)$ is the thick stable subcategory of $\Shv_{\In\lis}(S,R)$ generated by the sheaves of the form $R_S(V)$ when $V$ runs through the set of étale covers of $S$ and $\DM^{smA}_{\et,c}(S,R)$ is the thick stable subcategory of $\DM^{smA}_{\et}(S,R)$ generated by the motives of the form $M_S(V)$ when $V$ runs through the set of étale covers of $S$.

We now prove that the functor $$\Shv_{\In\lis}(S,R)\longrightarrow \DM^{smA}_{\et}(S,R)$$ induced by $\rho_!$ is an equivalence. Assuming that the functor $\rho_!$ is fully faithful, its essential surjectivity follows from the fact that $\mc{DM}^{smA}_{\et}(S,R)$ is generated as a localizing subcategory of itsef by the essential image of the functor $\rho_!$. Hence, it suffices to prove that the functor above is fully faithful.

Let $\rho^!$ be the right adjoint of $\rho_!$. The full faithfulness of $\rho_!$ holds for any good ring. It is equivalent to the fact that the ordinary transformation $\alpha\colon Id\rar \rho^!\rho_!$ is an equivalence. Using \Cref{conservativity,commutation}, we may assume that the ring $R$ is either of $\ell$-torsion where $\ell$ is invertible in $S$, or a $\Q$-algebra. The case where $R$ is of $\ell$-torsion follows from \cite[3.2]{bachmanrigidity}. Therefore, we can assume that the ring $R$ is a $\Q$-algebra.
 
The family of functors $\Map(R_S(W),-)$ when $W$ runs through the set of étale $S$-schemes is conservative. Therefore, it suffices to show that for any Ind-lisse étale sheaf $M$ and any étale $S$-scheme $W$, the canonical map $$\Map(R_S(W),M)\rar \Map(R_S(W),\rho^!\rho_!M)$$ is an equivalence. This is equivalent to the fact that the canonical map $$\Map(R_S(W),M)\rar \Map(M_S(W),\rho_!M)$$ is an equivalence.

Since $R$ is a $\Q$-algebra, by \cite[1.1.4, 1.1.9]{em} the sheaf $R_S(W)$ is a compact object of $\Shv(S_{\et},R)$ and the motive $M_S(W)$ is a compact object of $\DM_{\et}(S,R)$. As the stable category $\Shv_{\In\lis}(S,R)$ generated as a localizing subcategory of itself by objects of the form $R_S(V)$ where $V$ is finite étale over $S$, it therefore suffices to show that for any étale finite scheme $V$ over $S$ and any étale scheme $W$ over $S$, the map 
$$\rho_{W,V}\colon \Map(R_S(W),R_S(V))\rar \Map(M_S(W),M_S(V))$$
is an equivalence.

Recall now, that the étale sheaf $R_S(V)$ is dualizable with itself as a dual. As the functor $\rho_!$ is monoidal, this implies that the motive $M_S(V)$ is also dualizable with itself as a dual. Hence, the map $\rho_{W,V}$ is an equivalence if and only if the map $\rho_{W\times_S V, S}$ is an equivalence. As $$\rho_{X,S}=\bigoplus\limits_{T \in \pi_0(X)}\rho_{T,S},$$ it suffices to show that the maps $\rho_{T,S}$ are equivalences for $T$ étale over $S$ and connected. But if $T$ is étale over $S$, the map $\rho_{T,S}$ is an equivalence if and only if the canonical map $$\psi_T\colon \Map(R_T(T),R_T(T))\rar \Map(\un_T,\un_T)$$ is an equivalence.

Such a $T$ is irreducible and geometrically unibranch. Therefore, using \Cref{den}, we have 
\begin{align*}\pi_n(\Map(R_T(T),R_T(T)))&=\Hl^n_{\et}(T,R)\\
&=\begin{cases} R &\text{ if }n=0.\\
0 &\text{ otherwise}.
\end{cases}
\end{align*} 

Moreover, as such a $T$ is also regular (and connected), using \cite[B.5]{plh}, 
\begin{align*}\pi_n(\Map(\un_T,\un_T))&=\Hl^{n,0}_{\mc{M},\et}(T,R)\\
&=\begin{cases} R &\text{ if }n=0.\\
0 &\text{ otherwise}.
\end{cases}
\end{align*} 

Thus, it is sufficient to show that the map $\pi_0(\psi_T)$ is non-zero in order to show that it is an equivalence. But $\pi_0(\psi_T)$ is the natural map $$\Hom(R_T(T),R_T(T))\rar \Hom(\un_T,\un_T)$$ and this map sends the identity of $R_T(T)$ to the identity of $\un_T$.
\end{proof}
\begin{remark}\label{B7} Let $R$ be a regular good ring and let $S$ be a geometrically unibranch scheme which is the disjoint union of its irreducible components. Assume that the residue characteristic exponents of $S$ are invertible in $R$.

Then, the functor $\theta_!=L_{\AAA}\rho_{\#}$ (see \Cref{small et site}) induces monoidal equivalences 
$$\Shv_{\In\lis}(S,R)\longrightarrow \DA^{\eff,smA}_{\et}(S,R)$$
$$\Shv_{\lis}(S,R)\longrightarrow \DA^{\eff,smA}_{\et,c}(S,R)$$ 

Indeed, for the same reasons as in the proof of \Cref{smooth Artin h-motives}, we reduce to showing that the natural maps $$\Map(R_S(T),R_S(S))\rar \Map(L_{\AAA}R_S(T),L_{\AAA}R_S(S))$$ are equivalences for $T$ connected and étale over $S$ and $R$ a $\Q$-algebra (beware that on the right hand side, $R_S(T)$ and $R_S(S)$ are seen as sheaves on $\Sm_S$). This follows from \Cref{A1-locality} below.

Hence, if $S$ is regular, the functor $$\Sigma^\infty\colon\DA^{\eff,smA}_{\et}(S,R)\rar \DM^{smA}_{\et}(S,R)$$ is an equivalence. This is not true in general if $S$ is normal but non-regular since \cite[B.7]{plh} provides an example where étale cohomology in degree $n$ differs from motivic cohomology in degree $(n,0)$ for some integer $n$. In particular \Cref{smooth Artin h-motives} is false when the scheme $S$ is not regular.
\end{remark}

\begin{lemma}\label{A1-locality} Let $S$ be a scheme, let $X$ be a geometrically unibranch smooth $S$-scheme which is the disjoint union of its irreducible components and let $R$ be a regular good ring, where all residue characteristics of $X$ are invertible. Then, the étale sheaf $R_S(X)$ on $\Sm_S$ is $\AAA$-local.
\end{lemma}
\begin{proof} The functor $L_{\AAA}$ commutes with the functors $-\otimes_{\Z}\Q$ and $-\otimes_\Z \Z/p\Z$ by \Cref{commutation}. Using \cite[XV.2.2]{sga4} and \Cref{den}, \Cref{conservativity} applied to the canonical map $R_S(X)\rar L_{\AAA} R_S(X)$ yields the result.
\end{proof}

We now state some consequences of \Cref{smooth Artin h-motives}. 
\begin{corollary} Let $R$ be a regular good ring and let $S$ be a connected regular scheme. Assume that any residue characteristic exponent of $S$ is invertible in $R$. Then, the canonical map
$$\Hl^n_{\et}(S,R)\rar \Hl^{n,0}_{\mc{M},\et}(S,R)$$ is an isomorphism.
\end{corollary}
Combining \Cref{smooth Artin h-motives} with \Cref{loc cst,description des faisceaux lisses} yields the following statement.
\begin{corollary}\label{t-structure lisse} Let $R$ be a regular good ring and let $S$ be a regular scheme. Assume that the residue characteristic exponents of $S$ are invertible in $R$.
\begin{enumerate}
    \item The category $\DM^{smA}_{\et}(S,R)$ is endowed with a non-degenerate t-structure such that the functor $$\Shv_{\In\lis}(S,R)\longrightarrow \DM^{smA}_{\et}(S,R)$$ induced by the functor $\rho_!$ of \Cref{small et site 2} is t-exact when the left hand side is endowed with the ordinary t-structure. We still call this t-structure the \emph{ordinary t-structure}. Its heart is equivalent to the abelian category $\In \Loc(S,R)$. 
    \item The category $\DM^{smA}_{\et,c}(S,R)$ is endowed with a non-degenerate t-structure such that the functor $$\Shv_{\lis}(S,R)\longrightarrow \DM^{smA}_{\et,c}(S,R)$$ induced by the functor $\rho_!$ of \Cref{small et site 2} is t-exact when the left hand side is endowed with the ordinary t-structure. We still call this t-structure the \emph{ordinary t-structure}. Its heart is equivalent to the abelian category $\Loc_S(R)$. If $S$ is connected and $\xi$ is a geometric point of $S$, it is also equivalent to the abelian category $\Rep^A(\pi_1^{\et}(S,\xi),R)$.
\end{enumerate}
\end{corollary}

\begin{corollary}\label{smooth Artin h-motives in equal characteristic} Let $R$ be a regular good ring and let $S$ be a connected regular scheme of characteristic exponent $p$. Then, the functor $\rho_!$ of \Cref{small et site 2} induces monoidal equivalences
$$\Shv_{\In\lis}(S,R[1/p])\longrightarrow \DM^{smA}_{\et}(S,R)$$
$$\Shv_{\lis}(S,R[1/p])\longrightarrow \DM^{smA}_{\et,c}(S,R).$$
\end{corollary}
\begin{corollary}\label{Artin etale motives over a field} Let $k$ be a field of characteristic exponent $p$.  Let $R$ be a regular good ring.  Then, the functor $\rho_!$ induces monoidal equivalences 

$$\mc{D}(\mathrm{Mod}(G_k,R[1/p]))\simeq \Shv(k_{\et},R[1/p])\longrightarrow \DM_{\et}^{A}(k,R)$$
$$\mc{D}^b(\Rep^A(G_k,R[1/p]))\simeq \Shv_{\lis}(k,R[1/p]) \longrightarrow \DM_{\et,c}^{A}(k,R).$$
\end{corollary}
\begin{proof} Combine \Cref{smooth Artin h-motives in equal characteristic} with \Cref{lisse fields}.
\end{proof} 

\begin{remark}It is not true in general that the canonical functor $$\mc{D}^b(\Rep^A(\pi_1^{\et}(S,\xi),R))\rar \Shv_{\lis}(S,R)$$ is an equivalence. Indeed, letting $k$ be a separably closed field of characteristic $0$, the étale fundamental group of $\mb{P}^1_k$ is trivial. Therefore, the abelian category of its $\Z$-linear Artin representations is the category of finitely generated abelian groups, which for any geometric point $\xi$ yields $$\mathrm{Ext}^3_{\Rep^A\left(\pi_1^{\et}(\mb{P}^1_k,\xi),\Z\right)}(\Z,\Z)=0.$$ On the other hand, we have $$\Hl_{\et}^3(\mb{P}^1_k,\Z)=\mathrm{Ext}^3_{\Sh_{\et}(\mb{P}^1_k,\Z)}(\Z,\Z)$$ and this group is isomorphic to $\Q/\Z$ by \Cref{etale coh of curves}. Thus, the above functor cannot be an equivalence.

Finally, wen $R$ is a $\Q$-algebra, it was however shown in \cite[2.3.4]{ncd} that the above functor is an equivalence. 
\end{remark}

We end this section with a description of the $\ell$-adic realization functor on smooth Artin motives over a regular scheme.

\begin{proposition}\label{l-adic real 1} Let $S$ be a regular scheme and let $R$ be a localization of the ring of integers of a number field $K$, let $v$ be a non-archemedian valuation on the field $K$ and let $\ell$ be the prime number such that $v$ extends the $\ell$-adic valuation.  Assume that $\ell$ is invertible on $S$ and that all the residue characteristic exponents of $S$ are invertible in $R$. 
\begin{enumerate}\item If $v$ is non-negative on $R$, denote by $\lambda$ a uniformizer of $R_{(v)}$. Let $\Xi$ be the functor $$\Shv_{\In\lis}(S,R)\rar \Shv_{\lambda}(S,R_{(v)})$$ that sends an étale sheaf $M$ to the $\lambda$-complete sheaf $\lim (M\otimes_R R_{(v)})/\lambda^n$. 

Then, with the notations of \Cref{l-adic real}, the square
$$\begin{tikzcd}\Shv_{\In\lis}(S,R)\ar[d,"\Xi"] \ar[r,"\rho_!"] & \DM^{smA}_{\et}(S,R) \ar[d,"\rho_v"] \\
    \Shv_{\lambda}(S,R_{(v)}) \ar[r,"\Psi"]& \mc{D}(S,R_v)
\end{tikzcd}$$
is commutative.

Furthermore, when restricted to the subcategory of constructible smooth Artin étale motives, the $v$-adic realization functor is t-exact with respect to the ordinary t-structures 
and induces the exact functor $$-\otimes_R R_v\colon   \Loc_S(R)\rar \Sh^{smA}(S,R_v)$$
between the hearts. 

\item If $v$ is not non-negative on $R$, Let $A$ be the subring of $K$ on which the valuation $v$ is non-negative and let $\lambda$ be a uniformizer of $A$. Then, with the notations of \Cref{l-adic real}, the square
$$\begin{tikzcd}\Shv_{\lis}(S,R)\ar[d,"\Xi'"] \ar[r,"\rho_!"] & \DM^{smA}_{\et,c}(S,R) \ar[d,"\rho_v"] \\
    \Shv_{\lambda,c}(S,A)_{(0)} \ar[r,"\Psi'"]& \mc{D}^b_c(S,R_v)
\end{tikzcd}$$
is commutative using the notations of \Cref{p-loc}, if $\Psi'$ denotes the functor $\Psi_{(0)}$ and $\Xi'$ denotes the functor $\Xi_{(0)}\circ (-\otimes_R K)$.

Furthermore, the $v$-adic realization functor is t-exact with respect to the ordinary t-structures and induces the exact functor $$-\otimes_R R_v\colon   \Loc_S(R)\rar \Sh^{smA}(S,R_v)$$
between the hearts.  
\end{enumerate}
\end{proposition}
\begin{proof} The second assertion follows from the first assertion. We now prove the first assertion. 
Recall that with the notations of \Cref{l-adic real}, we have $$\rho_v=\Psi\circ \rho^! \circ \rho_\lambda^{\mathrm{CD}}.$$

But as $\rho^!$ is a right adjoint functor, it commutes with limits which yields $$\rho^! \circ \rho_\lambda^{\mathrm{CD}}\circ \rho_!=\Xi.$$ 

To show that the $\ell$-adic realization functor is t-exact when restricted to the subcategory of constructible smooth Artin étale motives, it suffices to show that the functor $$\Xi \circ \Psi\colon \Shv_{\lis}(S,R)\rar \mc{D}^b_c(S,R_v)$$ is t-exact. Let $M$ be a locally constant étale sheaf with coefficients in $R$. We have $$\Xi(\Psi(M))=\lim \nu^*(M)/\lambda^n.$$

It suffices to show that $\lim \nu^*(M)/\lambda^n$ lies in the heart of the ordinary t-structure on $\mc{D}^b_c(S,R_v)$. Now, by definition of the ordinary t-structure (see \cite[2.2.13]{bbd}) if $i$ is an integer, we have $$H^i(\lim \nu^*(M)/\lambda^n)=\lim H^i(\nu^*(M/\lambda^n))=\lim \nu^*\mathrm{Tor}_i^R(M,R/\lambda^n R).$$

Now, the exact sequence $$0\rar R\overset{\times \lambda^n}{\rar}R\rar R/\lambda^n R\rar 0$$ yields

$$\mathrm{Tor}_i^R(M,R/\lambda^n R)=\begin{cases} M/\lambda^n M &\text{if } i=0 \\
M[\lambda^n] &\text{if }i=-1 \\
0 &\text{otherwise}
\end{cases}$$

where $M[\lambda^n]$ denotes the $\lambda^n$-torsion of $M$. As the transition map $M[\lambda^{n+1}]\rar M[\lambda^n]$ is the multiplication by $\lambda$, we get $\lim M[\lambda^n]=0$. The result follows.
\end{proof}

\subsection{The Conservativity Conjecture for Artin Motives}

The conservativity conjecture, proved for Artin motives with rational coefficients in \cite{plh}, asserts that the $\ell$-adic realization functor \cite[7.2.11]{em} is conservative. With integral coefficients, we need to include realization functors with respect to every $\ell$. Indeed, if $M$ is an Artin motive of $p$-torsion and $\ell\neq p$, the $\ell$-adic realization of $M$ vanishes.

Furthermore, the stable category of constructible $\ell$-adic complexes has a non-pathological behavior only when $\ell$ is invertible on the base scheme. Therefore, we introduce the following functor.
\begin{definition}\label{reduced l-adic real} Let $S$ be a scheme, let $R$ be a localization of the ring of integers of a number field $K$, let $v$ be a valuation on $K$ and let $\ell$ be the prime number such that $v$ extends the $\ell$-adic valuation. 

The \emph{reduced $v$-adic realization functor} $\bar{\rho}_v$ is the composition 
$$\mc{DM}_{\et}(S,R)\overset{j_\ell^*}{\rar} \mc{DM}_{\et}(S[1/\ell],R)\overset{\rho_v}{\rar} \mc{D}(S[1/\ell],R_v),$$ where $j_\ell$ denotes the open immersion $S[1/\ell]\rar S$ and where the functor $\rho_v$ is the $v$-adic realization functor defined in \Cref{l-adic real}.

We still denote by $\bar{\rho}_v$ the restriction of this functor to the stable subcategory of constructible Artin motives.
\end{definition}

 A consequence of \Cref{smooth Artin h-motives} is the conservativity conjecture in the case of Artin motives.

\begin{proposition}\label{conservativityconj} Let $S$ be a scheme, let $R$ be a localization of the ring of integers of a number field $K$. If $v$ is a valuation on $K$, denote by $R_v$ the completion of $R$ with respect to $v$ and by $K_v$ the completion of $K$ with respect to $v$.
\begin{enumerate}\item If there is a prime number $\ell$ which is invertible on $S$, let $v$ be a valuation which extends the $\ell$-adic valuation then, the functor $$\bar{\rho}_v\colon\mc{DM}_{\et,c}^A(S,K)\rar \mc{D}(S,K_v)$$ is conservative. 
\item If $\ell$ and $\ell'$ are two distinct prime numbers, let $v$ be a valuation which extends the $\ell$-adic valuation and let $v'$ be a valuation which extends the $\ell'$-adic valuation. Then, when working with coefficients in $K$, the family $(\bar{\rho}_v,\bar{\rho}_{v'})$ is conservative.
\item The family $$\left(\bar{\rho}_v\colon\mc{DM}^A_{\et,c}(S,R)\rar \mc{D}^b_c(S[1/\ell(v)], R_v)\right)$$ where $v$ runs through the set of non-archimedian valuations on $K$ which are non-negative on $R$ and $\ell(v)$ denotes the prime number such that $v$ extends the $\ell$-adic valuation is conservative.
\end{enumerate}
\end{proposition}
\begin{proof} If $x$ is a point of $S$, denote by $i_x\colon \{ x\}\rar S$ the inclusion. The family of functors $(i_x^*)_{x \in S}$ is conservative by \cite[4.3.17]{tcmm}. In addition, the $v$-adic realization commutes with the six functors by \cite[7.2.16]{em}. Hence, to prove the proposition, we may assume that $S$ is the spectrum of a field $k$. Let $p$ be the characteristic exponent of $k$.

We prove the first assertion. Since $\ell\neq p$, we have $$\rho_v=\bar{\rho}_v.$$

Denote by $(BG_k)_{\mathrm{pro\acute{e}t}}$ the site of profinite continuous $G_k$-sets with covers given by continous surjections (see \cite[4.1.10]{bhatt-scholze}). Using \Cref{l-adic real 1}, the diagram
\[\begin{tikzcd}
	{\mathcal{D}^b(\mathrm{Rep}^A(G_k,K))} & {\Shv_{\mathrm{lisse}}(k,K)} & {\mathcal{DM}^A_{\mathrm{\acute{e}t}}(k,K)} \\
	{\mathcal{D}^b(\mathrm{Rep}(G_k,K_v))} & {\mathcal{D}(\Sh((BG_k)_{\mathrm{pro\acute{e}t}},K_v))} & {\mathcal{D}(k,K_v)}
	\arrow["{-\otimes_K K_v}"', from=1-1, to=2-1]
	\arrow[phantom, "\simeq", from=1-1, to=1-2]
	\arrow["{\rho_!}", from=1-2, to=1-3]
	\arrow["\rho_v",from=1-3, to=2-3]
	\arrow[phantom, "\simeq", from=2-2, to=2-3]
	\arrow["\alpha", from=2-1, to=2-2]
\end{tikzcd},\]
where $\alpha$ is the derived functor of the fully faithful exact functor $$\Rep(G_k,K_v)\rar \Sh((BG_k)_{\mathrm{pro\acute{e}t}},K_v),$$ is commutative.

Since the functors $\alpha$ and $-\otimes_K K_v$ are conservative and since the functor $\rho_!$ is an equivalence, the functor $\rho_v$ is conservative.

With the assumptions of the second assertion, one of the prime numbers $\ell$ and $\ell'$ is distinct from $p$ and the proof is the same.

We now prove the third assertion. Let $v$ be a valuation on $K$ such that $v(p)=0$ (\textit{i.e.} $v$ does not extend the $p$-adic valuation) Using \Cref{l-adic real 1}, the diagram
\[\begin{tikzcd}
	{\mathcal{D}^b(\mathrm{Rep}^A(G_k,R[1/p]))} & {\Shv_{\mathrm{lisse}}(k,R[1/p])} & {\mathcal{DM}^A_{\mathrm{\acute{e}t}}(k,R)} \\
	{\mathcal{D}^b(\mathrm{Rep}(G_k,R_v))} & {\mathcal{D}(\Sh((BG_k)_{\mathrm{pro\acute{e}t}},R_v))} & {\mathcal{D}(k,R_v)}
	\arrow["{-\otimes_R R_v}"', from=1-1, to=2-1]
	\arrow[phantom, "\simeq", from=1-1, to=1-2]
	\arrow["{\rho_!}", from=1-2, to=1-3]
	\arrow["\rho_v",from=1-3, to=2-3]
	\arrow[phantom, "\simeq", from=2-2, to=2-3]
	\arrow["\alpha", from=2-1, to=2-2]
\end{tikzcd},\]
where $\alpha$ is the derived functor of the fully faithful exact functor $$\Rep(G_k,R_v)\rar \Sh((BG_k)_{\mathrm{pro\acute{e}t}},R_v),$$ is commutative.


Now, if $M$ is a complex in $\mc{D}^b(\Rep^A(G_k,R[1/p])$ such that for any valuation $v$ such that $v(p)=0$, we have $$M\otimes_R R_v=0,$$  then, the complex $M$ vanishes. Hence, the family $(\overline{\rho}_\ell)$ is conservative.
\end{proof}

\subsection{Constructible Artin Motives as Stratified Representations}
Let $S$ be a scheme. Recall that a \emph{stratification} of $S$ is a finite partition of $S$ into equidimensional locally closed connected subsets of $S$ (called \emph{strata}) such that the topological closure of any stratum is a union of strata.




\begin{proposition}\label{AM.stratification} Let $S$ be a scheme, let $M$ be a constructible Artin étale motive over $S$. Then, there is a stratification of $S$ such that, for any stratum $T$,  the Artin motive $M|_T$ is smooth. Moreover if the scheme $S$ is excellent, we can assume the strata to be regular.
\end{proposition}
\begin{proof} We can assume that $S$ is reduced and connected. We proceed by noetherian induction on $S$. The local rings at the generic points of $S$ are fields. Let $M$ be a constructible Artin étale motive over $S$. As Artin étale motives over a field are smooth, by \Cref{continuity}, there is an open dense subscheme $U$ of $S$ such that $M|_U$ is smooth. If the scheme $S$ is excellent, we can furthermore assume that $U$ is regular. Using the induction hypothesis, there is a stratification of $S\setminus U$ such that for any stratum $T$, the motive $M|_T$ is a constructible smooth Artin motive and we are done.
\end{proof}

\begin{corollary}\label{stratification finitely many residue char} Let $S$ be an excellent scheme with finitely many residue characteristics and let $R$ be a regular good ring. Write $S_0=S\times_\Z \Q$ and $S_p=S\times_\Z \Z/p\Z$ for any prime number $p$. Then, we have prime numbers $p_1,\ldots,p_r$ such that $$S=S_0\sqcup S_{p_1} \sqcup \cdots \sqcup S_{p_r}.$$ 

Furthermore, the subscheme $S_0$ is open in $S$ and the subschemes $S_{p_1},\ldots,S_{p_r}$ are closed in $S$. 

Let $M$ be a constructible Artin étale motive over $S$. We have a stratification $\mc{S}_i$ of every $S_i$ when $i$ runs trhough the set $\{0,p_1,\ldots,p_r\}$ (and therefore a stratification of $S$), lisse étale sheaves $M_T\in \Shv_{\lis}(T,R)$ for $T$ in $ \mc{S}_0$ and lisse étale sheaves $M_T\in \Shv_{\lis}(T,R[1/p_i])$ for $T$ in $ \mc{S}_{p_i}$, such that for any stratum  $T$, we have $$M|_T=\rho_!(M_T).$$ 

In particular, the motive $M$ can be constructed from the $M_T$ by gluing using the localization triangle \eqref{AM.localization}.
\end{corollary}
\begin{corollary} Let $S$ be an excellent scheme and let $R$ be a regular good ring. Assume that all residue characteristics of $S$ are invertible in $R$.

Let $M$ be a constructible Artin étale motive over $S$. We have a stratification of $S$ and lisse étale sheaves $M_T\in \Shv_{\lis}(T,R)$ for every stratum $T$, such that we have $$M|_T=\rho_!(M_T).$$ 

In particular, the motive $M$ can be constructed from the $M_T$ by gluing using the localization triangle \eqref{AM.localization}.
\end{corollary}

\begin{corollary}\label{stratification Zp} Let $S$ be an excellent scheme, let $p$ be a prime number and let $R$ be a $\Z_{(p)}$-algebra.

Write $S_p=S\times_\Z \Z/p\Z$. Then, we have $$S=S[1/p]\sqcup S_{p}.$$

Furthermore the subscheme $S[1/p]$ is open in $S$ and the subscheme $S_{p}$ is closed in $S$. 

 Let $M$ be a constructible Artin étale motive over $S$ with coefficients in $R$. We have a stratification $\mc{S}[1/p]$ of $S[1/p]$ with regular strata, a stratification $\mc{S}_{p}$ of $S_{p}$ with regular strata, lisse étale sheaves $M_T\in \Shv_{\lis}(T,R[1/p])$ for $T$ in $\mc{S}_{p}$ and lisse étale sheaves $M_T\in \Shv_{\lis}(T,R)$ for $T\in \mc{S}[1/p]$, such that for any stratum  $T$, we have $$M|_T=\rho_!(M_T).$$ 

 In particular, the motive $M$ can be constructed from the $M_T$ by gluing using the localization triangle \eqref{AM.localization}.
\end{corollary}

\section{The Ordinary Homotopy \MakeLowercase{t}-structure on Artin Motives}\label{section 4}
\subsection{Definition for Non-constructible Artin Motives}
\begin{definition}\label{ordinary} Let $S$ be a scheme and let $R$ be a ring. The \emph{ordinary homotopy} t-structure on $\DM_{\et}^A(S,R)$ is the t-structure generated in the sense of \Cref{AM.t-structure generated} by the family of the motives $M_S^{\BM}(X)$ with $X$ quasi-finite over $S$.
\end{definition}
We denote with the small letters ${\ord}$ the notions related to the ordinary homotopy t-structure on $\DM_{\et}^A(S,R)$. For instance, we denote by ${}^{\ord}\tau_{\leqslant 0}$ the left truncation with respect to the ordinary homotopy t-structure and by $t_{\ord}$ the ordinary homotopy t-structure itself.
Using the same method as \Cref{AM.generators}, this t-structure is also generated by family of the motives $M_S^{\BM}(X)$ with $X$ finite (resp. étale) over $S$. Thus, an object $M$ of $\DM_{\et}^A(S,R)$ is in $\DM_{\et}^A(S,R)^{\geqslant n}$ if and only if for any étale $S$-scheme $X$, the complex $\Map_{\DM_{\et}(S,R)}(M_S(X),M)$ is $(n-1)$-connected.

\begin{proposition}\label{t-adj ord} Let $R$ be a ring. Let $f$ be a quasi-finite morphism of schemes and let $g$ be a morphism of schemes. Then, the functors $f_!$, $g^*$ and $-\otimes_S M$ are right $t_\mathrm{ord}$-exact.
\end{proposition}
\begin{proof} The full subcategory of those Artin motives $N$ such that $f_!N$, $g^*N$ and $N \otimes M$ are $t_{\ord}$-non-positive is closed under extensions, small coproducts and negative shifts and contains the set of generators of the t-structure. Thus, by definition of $t_{\ord}$, this subcategory contains any $t_{\ord}$-non-positive object. Hence, the functors $f_!$, $g^*$ and $- \otimes M$ are right $t_{\ord}$-exact.
\end{proof}

\begin{corollary}\label{AM.t-adj2 ord} Let $R$ be a ring and let $f$ be a morphism of schemes. Then, 
\begin{enumerate} 
\item If $f$ is étale, the functor $f^*=f^!$ is $t_{\ord}$-exact.
\item If $f$ is finite, the functor $f_!=f_*$ is $t_{\ord}$-exact.
\end{enumerate}
\begin{proof}
    If $f$ is étale, the functor $f^!$ is right adjoint to $f_!$ and therefore left $t_{\ord}$-exact. The case where $f$ is finite is handled similarly.
\end{proof}
\end{corollary}

\begin{lemma}\label{coeur} Let $S$ be a scheme and let $R$ be a ring. Let $X$ be finite over $S$. Then, the motive $h_S(X)$ lies in the heart of the ordinary homotopy t-structure.
\end{lemma}
\begin{proof} If $f$ is finite, the functor $f_*$ is $t_{\ord}$-exact by \Cref{AM.t-adj2 ord} and therefore, we may assume that $X=S$. By definition, we have $$\un_S\leqslant_{\ord} 0.$$ 

Let $U$ be étale over $S$, and let $n$ be an integer, by definition, we have $$\pi_n\Map_{\DM_{\et}(S,R)}(M_S(U),\un_S)=\Hl^{n,0}_{\mc{M},\et}(U,R).$$

But if the integer $n$ is negative, this group vanishes by \Cref{n<0 m<=0} and the result follows.
\end{proof}

We now explore the compatibility of the ordinary homotopy t-structure with respect to change of coefficients.

\begin{proposition}\label{coeff change ordinary} Let $S$ be a scheme, let $R$ be a ring, let $n$ be an integer and let $A$ be a localization of $\Z$. Denote by $\sigma_n\colon R\rar R/nR$ and by $\sigma_A\colon R\rar R\otimes_{\Z} A$ the natural ring morphisms. 
Recall the notations of \Cref{Change of coefficients}.
Then,
\begin{enumerate}\item The functors $(\sigma_A)_*$ and $(\sigma_n)_*$ are $t_{\ord}$-exact.
\item the functor $\sigma_A^*$ is $t_{\ord}$-exact.
\end{enumerate}

\end{proposition}
\begin{proof} This follows from \Cref{coeff change} and from \Cref{commutation}.
\end{proof}

\begin{proposition}\label{ordinary homotopy torsion} Let $S$ be a scheme and let $R$ be a ring. Recall \Cref{torsion motives are Artin}. Then, the inclusion functor $$\DM_{\et,\mathrm{tors}}(S,R)\rar \DM_{\et}^A(S,R)$$ is t-exact when the left hand side is endowed with the ordinary t-structure (see \Cref{ordinary torsion}) and when the right hand side is endowed with the ordinary homotopy t-structure.
\end{proposition}
\begin{proof} If $X$ is étale over $S$, the functor $\rho_!$ maps the étale sheaf $R_S(X)$ to the motive $M_S(X)$ which is $t_{\ord}$-non-positive. Recall that \Cref{generators ord} asserts that the sheaves of the form $R_S(X)$ with $X$ étale over $S$ generate the ordinary t-structure on the stable category $\Shv(S_{\et},R)$. Since the exact functor $$\rho_!\colon \Shv(S_{\et},R)\rar \DM_{\et}^A(S,R)$$ of \Cref{small et site 2} is compatible with small colimits, it is therefore right t-exact.

If $p$ is a prime number, we denote by $j_p\colon S[1/p]\rar S$ the open immersion. Recall that by \Cref{rigidity V1}, the functor $j_p^*$ induces an equivalence $$\DM_{\et,p^\infty-\mathrm{tors}}(S,R)\rar \DM_{\et,p^\infty-\mathrm{tors}}(S[1/p],R)$$ with inverse $(j_p)_!=(j_p)_*$.

Now, let $M$ be a torsion étale motive. Assume first that $M$ is t-non-positive with respect to the ordinary t-structure. By definition, this means that $$M=\bigoplus_{p\text{ prime}} (j_p)_!\rho_! N_p,$$ where for any prime number $p$, $N_p$ is a $p^\infty$-torsion étale sheaf which is t-non-positive with respect to the ordinary t-structure . 

By right t-exactness of the functor $\rho_!$, for any prime number $p$, the motive $\rho_!N_p$ is $t_{\ord}$-non-positive. Hence, using \Cref{t-adj ord}, so is the motive $(j_p)_!\rho_!N_p$. Therefore, the motive $M$ is $t_{\ord}$-non-positive.

Assume now that $M$ is t-non-negative with respect to the ordinary t-structure. Let $X$ be an étale $S$-scheme. By \Cref{torsion commutation}, we get $$\Map_{\DM_{\et}(S,R)}(M_S(X),M)=\bigoplus_{p\text{ prime}} \colim_n\Map_{\DM_{\et}(S,R)}(M_S(X),M\otimes_\Z \Z/p^n\Z[-1]).$$

If $p$ is a prime number and if $n$ is an integer, denote by $M_{p^n}$ the motive $M\otimes_\Z \Z/p^n\Z[-1]$. We have  $$\Map_{\DM_{\et}(S,R)}(M_S(X),M_{p^n})=\Map_{\DM_{\et}(S,R/nR)}(M_S(X,R/nR),M_{p^n}).$$

Since by \cite[A.3.4]{em}, we have $$M_{p^n}=(j_p)_*j_p^*M_{p^n},$$ we get
$$\Map_{\DM_{\et}(S,R/nR)}(M_S(X,R/nR),M_{p^n})=\Map_{\DM_{\et}(S[1/p]_{\et},R/nR)}(M_{S[1/p]}(X[1/p]),j_p^*(M_{p^n})).$$

And therefore, we get $$\Map_{\DM_{\et}(S,R)}(M_S(X),M_{p^n})=\Map_{\Shv(S[1/p]_{\et},R)}(R_{S[1/p]}(X[1/p]),\rho^!j_p^*(M_{p^n})).$$

If $p$ is a prime number and if $n$ is an integer, the $p^\infty$-torsion étale motive $M\otimes_\Z \Z/p^n\Z[-1]$ is t-non-negative with respect to the ordinary t-structure. Therefore, by definition of the ordinary t-structure, the $p^\infty$-torsion étale sheaf $\rho^!j_p^*(M_{p^n})$ is t-non-negative with respect to the ordinary t-structure. Therefore, the complex $\Map_{\Shv(S[1/p]_{\et},R)}(R_{S[1/p]}(X[1/p]),\rho^!j_p^*(M_{p^n}))$ is $(-1)$-connected. Hence, the complex $\Map_{\DM_{\et}(S,R)}(M_S(X),M)$ is also $(-1)$-connected. Hence, the motive $M$ is $t_{\ord}$-non-negative. 
\end{proof}

\subsection{The Ordinary Homotopy t-structure on Constructible Artin Motives and the t-exactness of the \texorpdfstring{$\ell$}{\textell}-adic Realization}

We begin this section with the case of smooth Artin motives when the residue characteristic exponents of the base scheme are invertible in the ring of coefficients.
\begin{proposition}\label{smooth Artin ordinary V1} Let $R$ be a regular good ring and let $S$ be a regular scheme. Assume that the residue characteristic exponents of $S$ are invertible in $R$.

Then, the functor $\rho_!$ of \Cref{small et site 2} induces a fully faithful t-exact functor $$\rho_!\colon \Shv_{\In\lis}(S,R)\rar \DM_{\et}^A(S,R)$$ when the left hand side is endowed with the ordinary t-structure and the right hand side is endowed with the ordinary homotopy t-structure.

In particular, the ordinary t-structure on the stable category $\DM^A_{\et}(S,R)$ induces a t-structure on the stable subcategories $\DM^{smA}_{\et}(S,R)$ and $\DM^{smA}_{\et,c}(S,R)$.
\end{proposition}
\begin{proof} In the proof of \Cref{ordinary homotopy torsion}, we showed that the exact functor $$\rho_!\colon \Shv(S_{\et},R)\rar \DM_{\et}^A(S,R)$$ is right t-exact. 
Therefore, it suffices to show that the exact functor $$\rho_!\colon \Shv_{\In\lis}(S,R)\rar \DM_{\et}^A(S,R)$$ is left t-exact. Let $N$ be an Ind-lisse étale sheaf and let $X$ be an étale $S$-scheme. By \Cref{smooth Artin h-motives}, we have $$\Map_{\DM_{\et}(S;R)}(M_S(X),\rho_!N)=\Map_{\Shv(S_{\et},R)}(R_S(X),N).$$

Hence, if the sheaf $N$ is t-non-negative, by \Cref{generatorze}, so is the motive $\rho_!N$. 
Therefore, the functor $$\rho_!\colon \Shv_{\In\lis}(S,R)\rar \DM_{\et}^A(S,R)$$ is t-exact. 

By \Cref{smooth Artin h-motives}, the functor $\rho_!$ is fully faithful with essential image the stable subcategory $\DM^{smA}_{\et}(S,R)$. Therefore, the ordinary homotopy t-structure on the stable category $\DM^A_{\et}(S,R)$ induces a t-structure on the stable subcategory $\DM^{smA}_{\et}(S,R)$.

Finally, since the ordinary t-structure on the stable category $\Shv(S_{\et},R)$ induces by \Cref{loc cst} a t-structure on the stable subcategory $\Shv_{\lis}(S,R)$, the ordinary homotopy t-structure on the stable category $\DM^{smA}_{\et}(S,R)$ also induces a t-structure on the stable subcategory $\DM^{smA}_{\et,c}(S,R)$ by \Cref{smooth Artin h-motives}.
\end{proof}

We want to show that the ordinary homotopy t-structure induces a t-structure on constructible Artin motives and to study the properties of the induced t-structure. In the case of Artin motives with rational coefficients, this was done in \cite{plh2}. We will use \Cref{ordinary homotopy torsion} to extend Pepin Lehalleur's result to arbitrary regular good rings. The first step is to extend Pepin Lehalleur's result on the t-exactness of the functor $f^*$ for any morphism of schemes $f$ to a result with intgeral coefficients. To that end, we need to introduce the following auxiliary notion.

\begin{definition} Let $S$ be a scheme and let $R$ be a ring. An Artin étale motive $M$ over $S$ with coefficients in $R$ is \emph{$\Q$-constructible} if the Artin motive $M\otimes_\Z \Q$ is constructible. We denote by $\DM_{\et,\Q-c}^A(S,R)$ the thick stable subcategory of the stable category $\DM_{et}^A(S,R)$ made of those objects which are $\Q$-constructible.
\end{definition}

Recall the following definition from \cite{plh2}.
\begin{definition} We say that a scheme $S$ allows resolution of singularities by alterations if for any separated $S$-scheme $X$ of finite type and any nowhere dense closed subset $Z$ of $X$, there is a proper alteration $g \colon X' \rar X$ with $X'$ regular and such that $g^{-1}(Z)$ is a strict normal crossing divisor.
\end{definition}

\begin{proposition}\label{Q-cons ord} Let $S$ be an excellent scheme allowing resolution of singularities by alterations and let $R$ be a regular good ring. The ordinary homotopy t-structure on the stable category $\DM_{\et}^A(S,R)$ induces a t-structure on the stable subcategory $\DM_{\et,\Q-c}^A(S,R)$.
\end{proposition}
\begin{proof} Assume first that the ring $R$ is a $\Q$-algebra. The proposition precisely asserts that the ordinary homotopy t-structure induces a t-structure on the stable category $\DM^A_{\et,c}(S,R)$. 

But by \cite[4.17]{plh}, the inclusion functor $$\DM^A_{\et}(S,R)\rar \DM^1_{\et}(S,R)$$ is t-exact when the left hand side is endowed with the ordinary homotopy t-structure and the right hand side is endowed with the standard motivic t-structure of Pepin Lehalleur defined in \cite[4.4]{plh}. By \cite[4.1]{plh2}, the truncation ${}^{\mathrm{std}}\tau_{\leqslant 0}$ with respect to the standard t-structure preserves compact objects. Therefore, the functor ${}^{\mathrm{std}}\tau_{\leqslant 0}={}^{{\ord}}\tau_{\leqslant 0}$ preserves $\DM^A_{\et,c}(S,R)$. Hence, using \cite[1.3.19]{bbd}, the ordinary homotopy t-structure induces a t-structure on the stable category $\DM^A_{\et,c}(S,R)$. 

Take now a general ring $R$. Let $M$ be a $\Q$-constructible Artin motive. \Cref{coeff change ordinary} implies that $${}^{{\ord}}\tau_{\leqslant 0}(M)\otimes_\Z \Q={}^{{\ord}}\tau_{\leqslant 0}(M\otimes_\Z \Q).$$

But the Artin motive ${}^{{\ord}}\tau_{\leqslant 0}(M\otimes_\Z \Q)$ is constructible using the result with coefficients in $R\otimes_\Z \Q$. Therefore, the Artin motive ${}^{{\ord}}\tau_{\leqslant 0}(M)$ is $\Q$-constructible. The result then follows from \cite[1.3.19]{bbd}.
\end{proof}

We can now formulate and prove an analog of Pepin Lehalleur's result on the t-exactness of the functor $f^*$ to for any morphism of schemes $f$.
\begin{proposition}\label{f^* ordinary} Let $S$ be an excellent scheme allowing resolution of singularities by alterations, let $f\colon T\rar S$ be a morphism of schemes and let $R$ be a ring.

Then, the functor $$f^*\colon \DM^A_{\et,\Q-c}(S,R)\rar \DM^A_{\et,\Q-c}(T,R)$$ is t-exact when both sides are endowed with the ordinary homotopy t-structure.
\end{proposition}
\begin{proof} \Cref{t-adj ord} implies that the functor $f^*$ is right $t_{\ord}$-exact. Hence, it suffices to show that it is left $t_{\ord}$-exact. 

Let $M$ be a $\Q$-constructible motive over $S$ which is $t_{\ord}$-non-negative. We have an exact triangle $$M\otimes_\Z \Q/\Z[-1]\rar M\rar M\otimes_\Z\Q .$$

By \Cref{coeff change ordinary}, the motive $M\otimes_\Z\Q$ is $t_{\ord}$-non-negative. Since the subcategory of $t_{\ord}$-non-negative objects is closed under limits, the motive $M\otimes_\Z \Q/\Z[-1]$ is also $t_{\ord}$-non-negative.

Furthermore, we have an exact triangle $$f^*(M\otimes_\Z \Q/\Z[-1])\rar f^*(M)\rar f^*(M\otimes_\Z\Q).$$ By \cite[4.1]{plh2}, the motive $f^*(M\otimes_\Z\Q)$ is $t_{\mathrm{std}}$-non-negative and therefore $t_{\ord}$-non-negative by \cite[4.17]{plh}. Furthermore, by \Cref{f^* ord torsion}, the motive $f^*(M\otimes_\Z \Q/\Z[-1])$ is $t_{\ord}$-non-negative. Since the subcategory of t-non-negative objects is closed under extensions, the motive $f^*(M)$ is also $t_{\ord}$-non-negative.
\end{proof}

Using, the above result, the t-exactness statement of \Cref{smooth Artin ordinary V1} for lisse sheaves and without the assumption that the residue characteristics are invertible in the ring of coefficients.

\begin{proposition}\label{smooth Artin ordinary V2}
Let $R$ be a regular good ring and let $S$ be a regular scheme.

Then, the functor $\rho_!$ of \Cref{small et site 2} induces a t-exact functor $$\rho_!\colon \Shv_{\lis}(S,R)\rar \DM_{\et}^A(S,R)$$ when the left hand side is endowed with the ordinary t-structure and the right hand side is endowed with the ordinary homotopy t-structure.
\end{proposition}
\begin{proof}
    The family $(i_x^*)_{x \in S}$ is conservative by \cite[4.3.17]{tcmm} and the functors $i_x^*$ are t-exact when $x$ is a point of $S$ by \Cref{f^* ordinary}. Therefore, by \Cref{cons family of t-ex functors} and \cite[4.4.2]{em}, we can assume that $S$ is the spectrum of a field $k$. In that case, by \Cref{Artin etale motives over a field}, letting $\sigma\colon R\rar R[1/p]$ be the canonical morphism, the down horizontal and right vertical arrows of the diagram
    $$\begin{tikzcd}
        \Shv_{\lis}(k,R) \ar[d,"\sigma^*"] \ar[r,"\rho_!"]& \DM^A_{\et}(k,R) \ar[d,"\sigma^*"]\\
        \Shv_{\lis}(k,R[1/p])\ar[r,"\rho_!"] & \DM^A_{\et}(k,R[1/p])
    \end{tikzcd}$$
    are equivalences; those equivalences are furthermore t-exact by \Cref{smooth Artin ordinary V1,coeff change ordinary}. Since the left vertical arrow is also t-exact, this finishes the proof. 
\end{proof}

We can now prove the main result of this section which extends \cite[4.1]{plh2}. 

\begin{theorem}\label{ordinary is trop bien}  Let $S$ be a scheme allowing resolution of singularities by alterations and let $R$ be a regular good ring. Then,
\begin{enumerate}\item The ordinary homotopy t-structure induces a t-structure on $\DM^A_{\et,c}(S,R)$. We still call this t-structure the \emph{ordinary homotopy} t-structure.
\item The objects of $\DM^A_{\et,c}(S,R)$ are bounded with respect to the ordinary homotopy t-structure which is therefore non-degenerate.
\item Let $f\colon T\rar S$ be a morphism of schemes. Then, the functor $$f^*\colon \DM_{\et,c}^A(S,R)\rar \DM_{\et,c}^A(T,R)$$ is $t_{\ord}$-exact. 
\item If $x$ is a point of $S$, denote by $i_x\colon \{ x\}\rar S$ the inclusion. The family $$\left( i^*_x:\DM_{\et,c}^A(S,R)\rar \DM_{\et,c}^A(k(x),R)\mid x \in S \right)$$ is a conservative family of $t_{\ord}$-exact functors. 
\item Assume that the ring $R$ is a localization of the ring of integers of a number field $K$. Let $v$ be a non-archimedian valuation on $K$ and let $\ell$ be the prime number such that $v$ extends the $\ell$-adic valuation. Then, the reduced $v$-adic realization functor 
$$\overline{\rho}_v\colon \DM_{\et,c}^A(S,R)\rar \mc{D}^b_c(S[1/\ell],R_v)$$ of \Cref{reduced l-adic real} is t-exact when the left hand side is endowed with the ordinary homotopy t-structure and the right hand side is endowed with the ordinary t-structure.
\end{enumerate}
\end{theorem}
\begin{proof} First notice that if the first assertion is true, the third assertion makes sense and follows from \Cref{f^* ordinary}. Furthermore, the third assertion implies the fourth assertion since the family $(i_x^*)_{x \in S}$ is conservative by \cite[4.3.17]{tcmm}. Using \Cref{cons family of t-ex functors}, the fourth assertion implies the fifth assertion. Finally the second assertion follows from \Cref{coeur}. Hence, it suffices to show the first assertion. 

If $\pp$ is a prime ideal of $\Z$ and if $X$ is a scheme, we say that an Artin motive is \emph{$\pp$-constructible} when, with the notations of \Cref{p-loc}, its image in the category $\DM^A_{\et}(X,R)_\pp$ belongs to $\DM^{A}_{\et,c}(X,R)_\pp$. Notice that \begin{itemize}
    \item if $f$ is a morphism of schemes, the functor $f^*$ preserve $\pp$-constructible Artin motives.
    \item if $f$ is a quasi-finite morphism of schemes, the functor $f_!$ preserve $\pp$-constructible Artin motives.
    \item the category of $\pp$-constructible motives over $X$ is a thick subcategory of $\DM^A_{\et}(X,R)$ for any scheme $X$.
\end{itemize}

Let $M$ be a constructible Artin motive over $S$. We want to show that the Artin motive ${}^{{\ord}}\tau_{\leqslant 0}(M)$ is constructible. By \Cref{p-locality}, it suffices to show that it is $\pp$-constructible for any maximal ideal $\pp$ of $\Z$.

Fix a maximal ideal $\pp$ of $\Z$. Let $p$ be a generator of $\pp$, let $S_p=S\times_{\Z} \Z/p\Z$, let $i_p\colon S_p\rar S$ be the immersion and let $j_p\colon S[1/p]\rar S$ be the complementary open immersion.

We have an exact triangle $$(j_p)_!j_p^*\hspace{1mm}{}^{{\ord}}\tau_{\leqslant 0}(M) \rar{}^{{\ord}} \tau_{\leqslant 0}(M) \rar (i_p)_*i_p^*\hspace{1mm}{}^{{\ord}}\tau_{\leqslant 0}(M).$$

The functors $(i_p)_*$ and $(j_p)_!$ preserve $\pp$-constructible Artin motives. Therefore, it suffices to show that the Artin motives $j_p^*\hspace{1mm}{}^{{\ord}}\tau_{\leqslant 0}(M)$ and $i_p^*\hspace{1mm}{}^{{\ord}}\tau_{\leqslant 0}(M)$ are $\pp$-constructible.

Using \Cref{f^* ordinary}, we have $$j_p^*\hspace{1mm}{}^{{\ord}}\tau_{\leqslant 0}(M)={}^{{\ord}}\tau_{\leqslant 0}(j_p^*M)$$ and $$i_p^*\hspace{1mm}{}^{{\ord}}\tau_{\leqslant 0}(M)={}^{{\ord}}\tau_{\leqslant 0}(i_p^*M).$$
Hence, we can assume that either $p$ is invertible on $S$ or that $S$ is of characteristic $p$. 

\Cref{AM.stratification} then gives a stratification of $S$ with regular strata such that for any stratum $i\colon T\rar S$, the motive $i^* M$ is a constructible smooth Artin motive. Assume that for any stratum $i\colon T\rar S$, the motive ${}^{{\ord}}\tau_{\leqslant 0}(i^*M)$ is $\pp$-constructible. Using the localization triangle and the stability of $\pp$-constructible objects under the functors $f^*$ and $f_!$ for $f$ quasi-finite, we can prove by induction on the number of strata that the motive ${}^{{\ord}}\tau_{\leqslant 0}(M)$ is constructible. Hence, we can assume that the scheme $S$ is regular and that the motive $M$ is a constructible smooth Artin motive.

If $p$ is invertible on $S$, the result follows from \Cref{p-loc lemma 1} below. 

Assume that $S$ is of characteristic $p$. It suffices to show that the Artin motive ${}^{{\ord}}\tau_{\leqslant 0}(M)$ is a constructible smooth Artin motive. Using \cite[A.3.4]{em}, we can assume that $p$ is invertible in $R$. The result then follows from \Cref{smooth Artin ordinary V1}.
\end{proof}

\begin{lemma}\label{p-loc lemma 1} Let $R$ be a regular good ring, let $S$ be a regular scheme, let $\pp$ be a maximal ideal of $\Z$ and let $p$ be a generator of $\pp$. Assume that $p$ is invertible in $S$. 
\begin{enumerate}
    \item The functor $\rho_!$ of \Cref{small et site 2} induces a monoidal equivalence
$$\Shv_{\lis}(S,R)_\pp\longrightarrow \DM^{smA}_{\et,c}(S,R)_\pp.$$
    \item The functor $\rho_!$ of \Cref{small et site 2} induces a t-exact fully faithful functor $$(\rho_!)_\pp\colon \Shv_{\lis}(S,R)_\pp\rar \DM_{\et}^A(S,R)_\pp$$ when the left hand side is endowed with the t-structure obtained by applying \Cref{p-loc 1} to the ordinary t-structure and the right hand side is endowed with the t-structure obtained by applying \Cref{p-loc 1} to the ordinary homotopy t-structure.

    In particular, the ordinary t-structure on the stable category $\DM^A_{\et}(S,R)_\pp$ induces a t-structure on the stable subcategory $\DM^{smA}_{\et}(S,R)_\pp$.
\end{enumerate}
\end{lemma} 
\begin{proof} \Cref{p-loc cons 2,p-loc cons 3} and \Cref{smooth Artin h-motives} imply the first assertion. The second assertion then follows from the first assertion and from \Cref{smooth Artin ordinary V2,p-loc 2}.
\end{proof}

\begin{definition} Let $S$ be a scheme allowing resolution of singularities by alterations and let $R$ be a regular good ring. The \emph{abelian category of ordinary Artin motives} $\mathrm{M}^A_{{\ord}}(S,R)$ is the heart of the ordinary homotopy t-structure of $\mc{DM}^A_{\et,c}(S,R)$.
\end{definition}

\begin{corollary} Let $S$ be a scheme allowing resolution of singularities by alterations and let $R$ be a regular good ring.  The following properties hold.
\begin{enumerate}\item Let $f\colon T\rar S$ be a morphism of schemes. Then, the functor $f^*$ induces an exact functor $$f^*\colon\mathrm{M}^A_{{\ord}}(S,R)\rar \mathrm{M}^A_{{\ord}}(T,R).$$
\item If $x$ is a point of $ S$, denote by $i_x\colon \{ x\}\rar S$ the inclusion and by $p(x)$ the characteristic exponent of $k(x)$. The family $$\left( \alpha^!i^*_x:\mathrm{M}^A_{{\ord}}(S,R)\rar \Rep^A(G_{k(x)},R[1/p(x)])\mid x \in S \right)$$ is a conservative family of exact functors.
\item Assume that the ring $R$ is a localization of the ring of integers of a number field $K$. Let $v$ be a non-archimedian valuation on $K$ and let $\ell$ be the prime number such that $v$ extends the $\ell$-adic valuation. Then, the reduced $v$-adic realization functor induces an exact and conservative functor
$$\overline{\rho}_v\colon\mathrm{M}^A_{{\ord}}(S,R)\rar \Sh^A(S[1/\ell],R_v).$$
\end{enumerate}
\end{corollary}

\bibliographystyle{alpha}
\bibliography{biblio.bib}
\end{document}